\documentclass{amsart}
\usepackage{graphicx}
\usepackage{url}
\usepackage{graphicx}
\usepackage{color}
\usepackage{times}
\usepackage{cite}
\usepackage{enumerate,latexsym}
\usepackage{latexsym}
\usepackage{amsmath,amssymb}
\usepackage{graphicx}
\usepackage{amsthm}
\usepackage{verbatim}
\newtheorem{thm}{Theorem}[section]
\newtheorem*{theorem*}{Theorem}
\newtheorem*{acknowledgement*}{Acknowledgements}
\newtheorem{cor}[thm]{Corollary}

\newtheorem{lem}[thm]{Lemma}
\newtheorem{prop}[thm]{Proposition}
\theoremstyle{definition}

\theoremstyle{remark}

\numberwithin{equation}{section}
\newcommand{\norm}[1]{\left\Vert#1\right\Vert}
\newcommand{\abs}[1]{\left\vert#1\right\vert}
\newcommand{\set}[1]{\left\{#1\right\}}
\newcommand{\Real}{\mathbb R}

\newcommand{\xX}[0]{\mathbf{x}}

\newcommand{\yY}[0]{\mathbf{y}}

\newcommand{\vV}[0]{\mathbf{v}}
\newcommand{\eE}[0]{\mathbf{e}}
\newcommand{\nN}[0]{\mathbf{n}}

\newcommand{\OO}{\mathbf{0}}

\title[Topology of Asymptotically Conical Self-Shrinkers]{A Topological Property of Asymptotically Conical Self-Shrinkers of Small Entropy}
\author{Jacob Bernstein}
\address{Department of Mathematics, Johns Hopkins University, 3400 N. Charles Street, Baltimore, MD 21218}
\email{bernstein@math.jhu.edu}
\author{Lu Wang}
\address{Department of Mathematics, University of Wisconsin-Madison, 480 Lincoln Drive, Madison, WI 53706}
\email{luwang@math.wisc.edu}
\thanks{The first author was partially supported by the NSF Grant DMS-1307953. The second author was partially supported by the Chapman Fellowship of the Imperial College London and by the AMS-Simons Travel Grant 2012-2014 and the NSF Grant DMS-1406240}
\begin{document}
\begin{abstract}
For any asymptotically conical self-shrinker with entropy less than or equal to that of a cylinder we show that the link of the asymptotic cone must separate the unit sphere into exactly two connected components, both diffeomorphic to the self-shrinker. Combining this with recent work of Brendle, we conclude that the round sphere uniquely minimizes the entropy among all non-flat two-dimensional self-shrinkers. This confirms a conjecture of Colding-Ilmanen-Minicozzi-White in dimension two.
\end{abstract}
\maketitle

\section{Introduction}
A hypersurface $\Sigma\subset \Real^{n+1}$ is said to be a \emph{self-shrinker} if it satisfies
\begin{equation} \label{SelfShrinkerEqn}
\mathbf{H}_\Sigma+\frac{\xX^\perp}{2}=\OO.
\end{equation}
Here $\mathbf{H}_\Sigma=-H_{\Sigma} \nN_\Sigma=\Delta_\Sigma \xX$ is the mean curvature vector of $\Sigma$ and $\xX^\perp$ is the normal component of the position vector. Self-shrinkers arise naturally in the mean curvature flow as the time slices of solutions that move self-similarly by scaling. Specifically, if $\Sigma$ is a self-shrinker, then
\begin{equation}
\set{\Sigma_t}_{t\in(-\infty,0)}=\set{\sqrt{-t}\, \Sigma}_{t\in(-\infty,0)}
\end{equation}
is a smooth mean curvature flow. Such self-similar flows are important as they model singularities of the mean curvature flow. From a variational point of view, self-shrinkers arise as critical points of the \emph{Gaussian (hyper)-surface area}
\begin{equation}
F[\Sigma]=\int_{\Sigma} \Phi(\xX) d\mathcal{H}^n = (4\pi)^{-\frac{n}{2}} \int_{\Sigma} e^{-\frac{|\xX|^2}{4}} d\mathcal{H}^n.
\end{equation}
Here $\mathcal{H}^n$ is $n$-dimensional Hausdorff measure and $\Phi$ is the Gaussian normalized so that $F[\Real^n\times \set{0}]=1$. Following Colding-Minicozzi \cite{CM}, the \emph{entropy} of a hypersurface is defined by
\begin{equation}
\lambda[\Sigma]=\sup_{(\yY,\rho)\in\Real^{n+1}\times\Real^+} F[\rho\Sigma+\yY].
\end{equation}
The Gaussian surface area and the entropy of self-shrinkers agree and so $\lambda[\Real^n\times\set{0}]=1$.

A hypersurface, $\Sigma$, is \emph{asymptotically conical}, if it is smoothly asymptotic to a regular cone; i.e., $\lim_{\rho\to 0} \rho\Sigma= \mathcal{C} (\Sigma)$ in $C^{\infty}_{loc}(\Real^{n+1}\setminus\set{\OO})$ for $\mathcal{C} (\Sigma)$ a regular cone. Our main result is a topological restriction on asymptotically conical self-shrinkers with small entropy:

\begin{thm} \label{MainACSThm}
Let $\Sigma$ be an asymptotically conical self-shrinker in $\Real^{n+1}$ for $n\geq 2$. If $\lambda[\Sigma]\leq\lambda[\mathbb{S}^{n-1}_*]$, then $\mathcal{L} (\Sigma)$, the link of the asymptotic cone $\mathcal{C} (\Sigma)$, separates $\mathbb{S}^{n}$ into two connected components both diffeomorphic to $\Sigma$. As a consequence, $\mathcal{L} (\Sigma)$ is connected.
\end{thm}

Here $\mathbb{S}^n\subset \Real^{n+1}$ is the unit sphere centered at the origin and, for $0\leq k\leq n$,
\begin{equation}
\mathbb{S}^{n-k}_{*}\times\Real^k=\set{(\xX,\yY)\in\Real^{n-k+1}\times\Real^k=\Real^{n+1}: |\xX|^2=2(n-k)}
\end{equation}
are the maximally symmetric self-shrinking cylinders with $k$-dimensional spine\footnote{In $\mathbb{R}^{n+1}$ a round sphere is a self-shrinker if and only if it is centered at the origin and its radius is $\sqrt{2n}$. In our notation, the radii of $\mathbb{S}^n_*$ and $\mathbb{S}^n$ are, respectively, $\sqrt{2n}$ and $1$. Thus, $\mathbb{S}^n_*$ is a self-shrinker while $\mathbb{S}^n$ is not.}. As the $\mathbb{S}^{n-k}_{*}\times\Real^k$ are self-shrinkers, their Gaussian surface area and entropy agree.  That is,
\begin{equation}
\lambda_n=\lambda [\mathbb{S}^n]=F[\mathbb{S}_*^n]=F[\mathbb{S}_*^n\times \Real^l].
\end{equation}
By a computation of Stone \cite{Stone},
\begin{equation}
2>\lambda_1>\frac{3}{2}>\lambda_2>\ldots>\lambda_n>\ldots\to\sqrt{2}.
\end{equation}

When $n=2$, Theorem \ref{MainACSThm} implies that $\Sigma$ is diffeomorphic to an open disk. This allows us to completely classify self-shrinkers in $\Real^3$ of small entropy.  Indeed, combining the work of Colding-Ilmanen-Minicozzi-White \cite{CIMW} (see also \cite{BernsteinWang}) and Brendle \cite{Brendle} with Theorem \ref{MainACSThm}, we conclude:

\begin{cor}\label{ShrinkerClassificationCor}
If $\Sigma\subset\Real^{3}$ is a self-shrinker with $\lambda[\Sigma]\leq\lambda_1$, then $\Sigma$ is either $\mathbb{S}^2_*$, or some rotation of  $\Real^2\times\set{0}$ or $\mathbb{S}^1_*\times\Real$. In fact, there is a $\delta_0>0$ so that these are the only self-shrinkers with entropy less than $\lambda_1+\delta_0$.
\end{cor}

This answers affirmatively the $n=2$ case of \cite[Conjecture 0.10]{CIMW} which asks whether  $\mathbb{S}^n_*$ minimizes the entropy amongst all non-flat self-shrinkers. In addition, this gives, for mean curvature flows in $\Real^3$, that $\lambda_2$ is the best constant in White's formulation \cite{WhiteReg} of Brakke's regularity theorem \cite{B}. Another simple consequence of Corollary \ref{ShrinkerClassificationCor} is a sharp lower bound for the entropy of all closed surfaces in $\Real^3$ (see \cite[Theorem 1.1]{BernsteinWang} for a different approach which works for all $2\leq n\leq 6$) and a new entropy lower bound for all closed surfaces in $\Real^3$ of positive genus:

\begin{cor} \label{TopSurfaceCor}
If $\Sigma$ is a closed surface in $\Real^3$, then $\lambda[\Sigma]\geq\lambda_2$ with equality if and only if $\Sigma$ is, up to translations and scalings, $\mathbb{S}^2_*$. Furthermore, if $\Sigma$ has positive genus, then $\lambda[\Sigma]>\lambda_1$. 
\end{cor}

Our proof of Theorem \ref{MainACSThm} parallels that of Colding-Ilmanen-Minicozzi-White, who, in \cite{CIMW}, showed that $\mathbb{S}^n_*$ uniquely minimizes the entropy within the class of closed self-shrinkers. Indeed, our arguments may be thought of as a natural analog, in the asymptotically conical setting, of their arguments.  However, we wish to emphasize one crucial difference between our proofs. Namely, in \cite{CIMW} the authors exploit the fact that a closed hypersurface must form a singularity in finite time -- indeed, they point out this is the main obstruction to extending their result to the non-compact setting.  We, however, make use of the fact that flows of asympotically conical surfaces of small entropy must exist without singularities for long-time to show that the flows eventually become star-shaped with respect to the origin.

More precisely, inspired by \cite{CIMW}, we first study properties of a quantity along a smooth mean curvature flow that we call the \emph{shrinker mean curvature} -- see \eqref{SMCDefn}. In particular, we use a parabolic maximum principle on non-compact manifolds as in \cite{EH} to show that if the initial surface of a smooth mean curvature flow is well-behaved at infinity and has positive shrinker mean curvature, then the shrinker mean curvature remains positive along the flow -- see Proposition \ref{LowBndSMCProp}. We also use a standard parabolic maximum principle to conclude that for such a flow, the shrinker mean curvature controls the second fundamental form -- see Proposition \ref{CurvBndSMCProp}.   

A variant on the shrinker mean curvature is also considered in \cite{CIMW} where the above properties are shown to hold for smooth mean curvature flows of \emph{closed} hypersurfaces. However, in \cite{CIMW} the authors find it convenient to think of self-shrinkers as static points and so instead work with a certain rescaled mean curvature flow.  As such, their quantity has a different, but related, form.  While this may appear to be merely a technical distinction it actually holds the key to our proof.  Indeed, the quantity we consider makes sense for smooth mean curvature flows which start close to a self-shrinker but that persist up to (and beyond) the singular time of the self-shrinker.  By its construction, this cannot be true of the quantity considered in \cite{CIMW}.

 We next observe that if $\Sigma$ is any non-flat asymptotically conical self-shrinker, then there are asymptotically conical perturbations $\Gamma^\pm$ of $\Sigma$ on each side of $\Sigma$ which have positive shrinker mean curvature (relative to the correct orientation), strictly smaller entropy and which have the same asymptotic cone as that of $\Sigma$ -- see Proposition \ref{PerturbedInitialSurfaceProp}.  As in \cite{CIMW}, these two hypersurfaces are found by considering the normal exponential graphs of small multiples of the lowest eigenfunction of the (self-shrinker) stability operator. Again the non-compactness introduces certain technical difficulties.

In our third step, we consider the smooth mean curvature flows with initial surfaces $\Gamma^\pm$. Directly using arguments in \cite{CIMW}, we prove that for non-compact flows of positive shrinker mean curvature if a singularity develops, then the Gaussian density of the flows at the singular point must be at least $\lambda_{n-1}$. Hence, if the entropy of $\Sigma$ is at most $\lambda_{n-1}$, then the flows starting from $\Gamma^\pm$ never develop singularities. We further show that each time slice of these flows is smoothly asymptotic to the same cone as that of $\Sigma$.

Finally, we observe that any long-time solution of the mean curvature flow with positive shrinker mean curvature and which is asymptotically conical must have a time-slice which is star-shaped -- see Theorem \ref{StarshapThm} and Proposition  \ref{RadialGraphTopProp}. Theorem \ref{MainACSThm} is then an immediate consequence. The proofs of Corollaries \ref{ShrinkerClassificationCor} and \ref{TopSurfaceCor} are  straightforward.

\section{Notation} 
Let $\Real^{n+1}$ denote the standard $(n+1)$-dimensional Euclidean space. We denote by 
\begin{equation}
B_R(\xX_0)=\set{\xX\in\Real^{n+1}: |\xX-\xX_0|<R}
\end{equation}
the open Euclidean ball of radius $R$ centered at $\xX_0$. When the center is the origin $\OO$, we will omit it. A \emph{hypersurface} in $\Real^{n+1}$ is a proper codimension one submanifold of $\Real^{n+1}$. At times it will be convenient to distinguish between a point $p\in\Sigma$ and its \emph{position vector} $\xX(p)\in\Real^{n+1}$.  Given a set $\Omega\subset\Real^{n+1}$ we define the $\rho$-tubular neighborhood of $\Omega$ to be the open set
\begin{equation}
\mathcal{T}_{\rho} \left(\Omega\right)=\bigcup_{p\in\Omega} B_{\rho}(p)\subset \Real^{n+1}.
\end{equation} 

Any hypersurface is two-sided and hence orientable -- see for instance \cite{Samelson}. For this reason, we will always orient a hypersurface, $\Sigma$, by a choice of unit normal 
\begin{equation}
\mathbf{n}_\Sigma:\Sigma \to \mathbb{S}^n\subset \Real^{n+1}.
\end{equation}  
For such a $\Sigma$, we let $g_{\Sigma}$ denote the induced Riemannian metric, $A_\Sigma$ be the second fundamental form (with respect to the choice of unit normal $\nN_\Sigma$) and $H_\Sigma$ the associated scalar mean curvature which we take to be the trace of $A_\Sigma$. Observe that the mean curvature vector $\mathbf{H}_{\Sigma}=-H_{\Sigma} \nN_\Sigma$ is well-defined independent of choices of unit normal as is the norm $|A_\Sigma|$ of the second fundamental form. For any hypersurface, $\Sigma$, and $\rho\in\Real^{+}$ we have the natural smooth map
\begin{equation}
\Psi_\rho: \Sigma\times (-\rho, \rho)\to \mathcal{T}_\rho(\Sigma)
\end{equation}
given by $\Psi_\rho(p, s)=\xX(p)+s\nN_\Sigma(p)$. We say $\mathcal{T}_\rho(\Sigma)$ is a \emph{regular tubular neighborhood} if this map is a diffeomorphism.

We consider space-time, $\Real^{n+1}\times\Real$, to be the set of space-time points $X=(\xX,t)\in\Real^{n+1}\times\Real$. And let $O=(\OO,0)$ be the space-time origin. We will often focus on the subset of space-time consisting of space-time points with negative time which we denote by $\Real^{n+1}\times \Real^-$. Given $R, \tau>0$, let 
\begin{equation}
C_{R,\tau} (X_0)=\set{(\xX,t)\in\Real^{n+1}\times\Real: \xX\in B_R(\xX_0), |t-t_0|<\tau}
\end{equation}
be the parabolic cylinder of radius $R$ and height $\tau$ centered at $X_0=(\xX_0,t_0)$.  It will also be convenient to consider the backward parabolic cylinder
\begin{equation}
C_{R,\tau}^- (X_0)=C_{R,\tau}(X_0)\cap\Real^{n+1}\times\set{t<t_0}.
\end{equation}
The \emph{parabolic boundary} of a parabolic cylinder is defined to be
\begin{equation}
\partial_P C_{R,\tau} (X_0) =\partial C_{R,\tau} (X_0)\setminus \left(B_R(\xX_0)\times\set{t=t_0+\tau} \right),
\end{equation}
where $\partial$ is the topological boundary. Likewise, $\partial_P C_{R,\tau}^-(X_0)=\partial C_{R,\tau}^-(X_0)\cap\partial_P C_{R,\tau}(X_0)$.

A \emph{smooth mean curvature flow} is a collection of hypersurfaces $\set{\Sigma_t}_{t\in I}$, where $I$ is some interval in $\Real$ and so that there is a smooth map $F: M\times I\to\Real^{n+1}$ so that for each $t\in I$, $F(\cdot,t): M\to \Sigma_t\subset\Real^{n+1}$ is a parameterization of $\Sigma_t$ and so that
\begin{equation} \label{NormalMCFEqn}
\left(\frac{\partial}{\partial t}F(p,t)\right)^\perp =\mathbf{H}_{\Sigma_t}(F(p,t)).
\end{equation}
We will always take the hypersurfaces $\Sigma_t$ in a smooth mean curvature flow to be oriented so that the unit normal is smooth in $t$. It is often convenient to consider the \emph{space-time track} of a smooth mean curvature flow
\begin{equation}
\mathcal{S}=\set{(\xX(p),t)\in\mathbb{R}^{n+1}\times \mathbb{R}: p\in\Sigma_t},
\end{equation}
which is a smooth submanifold of space-time (with boundary if $I$ contains either of its endpoints) that is transverse to each constant time hyperplane. We will not distinguish between a smooth mean curvature flow $\set{\Sigma_t}_{t\in I}$ and its space-time track and so denote both by $\mathcal{S}$. Along the space-time track of a smooth mean curvature flow $\mathcal{S}$,  let $\frac{d}{dt}$ be the smooth vector field given by 
\begin{equation}
\frac{d}{dt}=\frac{\partial}{\partial t}+\mathbf{H}_{\Sigma_t}.
\end{equation}
It is not hard to see that this vector field is tangent to $\mathcal{S}$ and the position vector satisfies
\begin{equation} \label{MCFEqn}
\frac{d}{dt} \xX(p,t)=\mathbf{H}_{\Sigma_t}(p).
\end{equation}
It is straightforward (and standard) to compute the evolution of various geometric quantities with respect to this vector field -- see for instance \cite[Appendix B]{EckerBook}. 

For a set $\Omega\subset \Real^{n+1}$, an $\xX\in \Real^{n+1}$ and a $\rho\in\Real^+$, let
\begin{enumerate}
 \item $ \Omega+\xX=\set{\yY\in\Real^{n+1}: \yY-\xX\in\Omega}$, the translation of $\Omega$ by $\xX$; and
\item $\rho\, \Omega= \set{\yY\in\Real^{n+1}: \rho^{-1} \yY\in\Omega}$, the scaling of $\Omega$ by $\rho$. 
\end{enumerate}
Similarly, for a set $\Omega\subset\Real^{n+1}\times\Real$ in space-time, an $X\in\Real^{n+1}\times\Real$ and a $\rho\in\Real^+$,
\begin{enumerate}
 \item $ \Omega+X=\set{Y\in\Real^{n+1}\times \Real: Y-X\in\Omega}$, the space-time translation of $\Omega$ by $X$; and
 \item $\rho\, \Omega= \set{(\yY,t)\in \Real^{n+1}: (\rho^{-1}\yY , \rho^{-2} t) \in\Omega}$, the parabolic scaling of $\Omega$ by $\rho$. 
\end{enumerate}
Clearly, if $\mathcal{S}$ is the space-time track of a smooth mean curvature flow, then so is $\rho \mathcal{S}+X$ for any $\rho\in\Real^+$ and $X\in \Real^{n+1}\times\Real$.

\section{Shrinker Mean-Convexity}
Let $\mathcal{S}=\set{\Sigma_t}_{t\in [-1,T)}$ be a smooth mean curvature flow. Along the flow $\mathcal{S}$, we define the \emph{shrinker mean curvature relative to the space-time point $X_0=(\xX_0,t_0)$} to be
\begin{equation}\label{SMCDefn}
S^{X_0}_{\Sigma_t} (p)=2(t_0-t) H_{\Sigma_t} (p)-\left(\xX(p)-\xX_0\right)\cdot\nN_{\Sigma_t} (p).
\end{equation}
We emphasize that, due to the $t$ dependence, this quantity is defined for $\mathcal{S}$. Given a time $t\in\Real$ and hypersurface $\Sigma$, the shrinker mean curvature of $\Sigma$ relative to the space-time point $X_0$ and time $t$ is defined to be
\begin{equation}
S^{X_0,t}_{\Sigma} (p)=2(t_0-t) H_{\Sigma} (p)-\left(\xX(p)-\xX_0\right)\cdot\nN_{\Sigma} (p).
\end{equation}

We recall the following evolution equation for the shrinker mean curvature due to Smoczyk \cite[Proposition 4]{Smoczyk}.

\begin{lem} \label{SMCEvLem} 
Along a mean curvature flow $\set{\Sigma_t}_{t\in I}$, the shrinker mean curvature relative to $X_0$ satisfies
\begin{equation}\label{SMCEvEqn}
\frac{d}{dt} S^{X_0}_{\Sigma_t} =\Delta_{\Sigma_t} S^{X_0}_{\Sigma_t} +\abs{A_{\Sigma_t}}^2 S^{X_0}_{\Sigma_t}.
\end{equation}
\end{lem}

Observe that $\mathcal{S}$ is self-similar with respect to parabolic rescalings about $X_0$ if and only if $S^{X_0}_{\Sigma_t}\equiv 0$. More generally, if one parabolically dilates $\mathcal{S}$ about $X_0$, then the vector field of the normal variation of this family at $\mathcal{S}$ is the shrinker mean curvature vector relative to $X_0$. Hence, as parabolic dilations are symmetries of the mean curvature flow, \eqref{SMCEvEqn} may be viewed as the linearization of the mean curvature flow. In a similar fashion, the quantities $\eE_i\cdot \nN_{\Sigma_t}$ used by Ecker-Huisken \cite{EH} are generated by spatial translations and so also satisfy \eqref{SMCEvEqn}. Likewise, the mean curvature $H_{\Sigma_t}$ which also satisfies \eqref{SMCEvEqn} arises in this manner from temporal translations.

We now use the maximum principle of Ecker-Huisken \cite[Corollary 1.1]{EH} on non-compact manifolds to show that if a smooth mean curvature flow $\set{\Sigma_t}_{t\in [-1,T)}$ satisfies that $\Sigma_{-1}$ is shrinker mean convex with respect to $X_0$, then this remains true for all $t\in (-1,T)$.

\begin{prop}\label{LowBndSMCProp}
Let $\mathcal{S}=\set{\Sigma_t}_{t\in [-1,T)}$ be a smooth mean curvature flow in $\Real^{n+1}$ with finite entropy. Suppose that $\Sigma_{-1}$ satisfies
\begin{equation} \label{InitialSMCEqn}
S^{X_0}_{\Sigma_{-1}}(p)\geq c\left(1+|\xX (p)|^2\right)^{-\alpha}
\end{equation}
for some $c>0$ and $\alpha\geq 0$, and that
\begin{equation} \label{C3BndEqn}
M=\sup_{t\in [-1,T)}\sup_{\Sigma_t\setminus B_R} \abs{A_{\Sigma_t}}+\abs{\nabla_{\Sigma_t}A_{\Sigma_t}}+\abs{\nabla^2_{\Sigma_t} A_{\Sigma_t}}<\infty
\end{equation}  
for some $R>0$. For all $(p,t)\in\mathcal{S}$,
\begin{equation} \label{SMCFlowEqn}
S^{X_0}_{\Sigma_t}(p)\geq c\left(1+|\xX (p)|^2+2n (t+1)\right)^{-\alpha}. 
\end{equation}
\end{prop}

\begin{proof}
Let 
\begin{equation}
\eta(\xX,t)=1+ |\xX|^2+2n (t+1).
\end{equation}
By \cite[Lemma 1.1]{EH} we have that
\begin{equation} \label{BarrierEvEqn}
\left(\frac{d}{dt} -\Delta_{\Sigma_t} \right) \eta=0,
\end{equation}
and so
\begin{equation}
\left(\frac{d}{dt} -\Delta_{\Sigma_t} \right) \eta^{\alpha}=-\alpha (\alpha-1) |\nabla_{\Sigma_t}\log\eta|^2 \eta^{\alpha}.
\end{equation}

Define 
\begin{equation}
u(p,t)=\eta^\alpha(\xX(p),t) S^{X_0}_{\Sigma_t}(p) \quad\mbox{for $(p,t)\in\mathcal{S}$}.
\end{equation}
Then it follows from Lemma \ref{SMCEvLem} and \eqref{BarrierEvEqn} that
\begin{equation} \label{EQNforu}
\left(\frac{d}{dt}-\Delta_{\Sigma_t}\right) u+2\alpha\nabla_{\Sigma_t}\log\eta\cdot\nabla_{\Sigma_t} u=\abs{A_{\Sigma_t}}^2 u+\alpha (\alpha+1) \abs{\nabla_{\Sigma_t} \log\eta}^2 u.
\end{equation}
Notice that 
\begin{equation}
\abs{\nabla_{\Sigma_t}\log\eta} (p,t)=\frac{2 \abs{\xX(p)^\top}}{1+\abs{\xX(p)}^2+2n (t+1)}<2,
\end{equation}
and by \eqref{C3BndEqn} and \eqref{EQNforu} we have that on $\mathcal{S}\setminus C_{R,T+1}(\OO,-1)$,
\begin{equation}
\abs{\frac{du}{dt}}(p,t)+\sum_{i=0}^2\abs{\nabla^i_{\Sigma_t} u} (p,t)\leq C(M,\alpha, X_0) \left(1+|\xX (p)|^2+2n (t+1)\right)^{\alpha+1}.
\end{equation}
Invoking \eqref{InitialSMCEqn} and finiteness of the entropy, Theorem \ref{NonCompactMaxPrincipleThm} implies, when $R=0$, that 
\begin{equation}
\inf_{p\in\Sigma_t} u(p,t)\geq\inf_{p\in\Sigma_{-1}} u(p,-1)\geq c,
\end{equation}
giving immediately \eqref{SMCFlowEqn}.
\end{proof}

We further adapt some ideas of \cite{EH} -- specifically the proof of \cite[Lemma 4.1]{EH} -- to prove a relationship between the shrinker mean curvature and the second fundamental form for shrinker mean convex flows.

\begin{prop} \label{CurvBndSMCProp}
Let $\mathcal{S}=\set{\Sigma_t}_{t\in [-1,T)}$ be a smooth mean curvature flow in $\Real^{n+1}$ with finite entropy. Suppose that $\Sigma_{-1}$ satisfies 
\begin{equation}
S^{X_0}_{\Sigma_{-1}}(p)\geq c \left(1+|\xX(p)|^2\right)^{-\alpha}
\end{equation}
for some $c>0$ and $\alpha\geq 0$, and that
\begin{equation} \label{InitialCurvBndEqn}
\tilde{M}=\max\set{\sup_{t\in [-1,T)}\sup_{ \Sigma_t\setminus B_{R}} \abs{A_{\Sigma_t}}+\abs{\nabla_{\Sigma_t} A_{\Sigma_t}}+\abs{\nabla^2_{\Sigma_t} A_{\Sigma_t}}, \sup_{\Sigma_{-1}} \abs{A_{\Sigma_{-1}}}}<\infty
\end{equation}
for some $R>0$. Then for all $(p,t)\in\mathcal{S}$,
\begin{equation} \label{CurvBndSMCEqn}
\abs{A_{\Sigma_t} (p)}\leq c^{-1} \tilde{M} \left(1+|\xX(p)|^2+R^2+2n (t+1)\right)^\alpha S^{X_0}_{\Sigma_t}(p). 
\end{equation}
\end{prop}

\begin{proof}
First observe that, by Proposition \ref{LowBndSMCProp}, for all $(p,t)\in\mathcal{S}$,
\begin{equation} \label{LowSMCFlowEqn}
S^{X_0}_{\Sigma_t}(p)\geq c \left(1+|\xX(p)|^2+2n (t+1)\right)^{-\alpha}.
\end{equation}
Thus, in view of \eqref{InitialCurvBndEqn}, it suffices to prove \eqref{CurvBndSMCEqn} on $\mathcal{S}\cap C_{R,T+1}^-(\OO,T)$.

Let
\begin{equation} 
u(p,t)=\abs{A_{\Sigma_t}(p)}^2 v^2(p,t)=\abs{A_{\Sigma_t}(p)}^2 |S^{X_0}_{\Sigma_t}(p)|^{-2}.
\end{equation}  
By \cite[Appendix B, (B.9)]{EckerBook},   
\begin{equation}
\left(\frac{d}{dt} -\Delta_{\Sigma_t} \right) \abs{A_{\Sigma_t}}^2 \leq -2 \abs{\nabla_{\Sigma_t} \abs{A_{\Sigma_t}} }^2 +2 \abs{A_{\Sigma_t}}^4.
\end{equation}
And by Lemma \ref{SMCEvLem}, 
\begin{equation}
\left(\frac{d}{dt}-\Delta_{\Sigma_t}\right) v^2=-2 \abs{A_{\Sigma_t}}^2 v^2-6 \abs{\nabla_{\Sigma_t} v}^2
\end{equation}
Thus, we compute the evolution equation of $u$:
\begin{equation}
\begin{split}
\left(\frac{d}{dt}-\Delta_{\Sigma_t} \right) u = & \abs{A_{\Sigma_t}}^2\left(\frac{d}{dt} -\Delta_{\Sigma_t} \right) v^2 -2 \nabla_{\Sigma_t} v^2 \cdot \nabla_{\Sigma_t} \abs{A_{\Sigma_t}}^2 \\
& +v^2 \left(\frac{d}{dt}-\Delta_{\Sigma_t} \right) \abs{A_{\Sigma_t}}^2 \\
& \leq -6 \abs{A_{\Sigma_t}}^2 \abs{\nabla_{\Sigma_t} v}^2-2 \nabla_{\Sigma_t} v^2 \cdot \nabla_{\Sigma_t} \abs{A_{\Sigma_t}}^2 \\
& -2 \abs{\nabla_{\Sigma_t} \abs{A_{\Sigma_t}} }^2 v^2.
\end{split}
\end{equation}
Using Young's inequality, we obtain the estimate: 
\begin{equation}
\begin{split}
 -2 \nabla_{\Sigma_t} v^2 \cdot \nabla_{\Sigma_t} \abs{A_{\Sigma_t}}^2 = & -\frac{\nabla_{\Sigma_t} v^2}{v^2} \cdot \nabla_{\Sigma_t} \left(\abs{A_{\Sigma_t}}^2 v^2\right) +\frac{\abs{A_{\Sigma_t}}^2 \abs{\nabla_{\Sigma_t} v^2}^2}{v^2} \\
    & -\nabla_{\Sigma_t} v^2 \cdot \nabla_{\Sigma_t} \abs{A_{\Sigma_t}}^2\\
 = & -\frac{\nabla_{\Sigma_t} v^2}{v^2}\cdot \nabla_{\Sigma_t} u +4 \abs{A_{\Sigma_t}}^2 \abs{\nabla_{\Sigma_t} v}^2 \\
    & -4v \abs{A_{\Sigma_t}} \nabla_{\Sigma_t} v \cdot \nabla_{\Sigma_t} \abs{A_{\Sigma_t}} \\
\leq & -\frac{\nabla_{\Sigma_t} v^2}{v^2}\cdot \nabla_{\Sigma_t} u +6 \abs{A_{\Sigma_t}}^2 \abs{\nabla_{\Sigma_t} v}^2  \\
       &+2v^2 \abs{\nabla_{\Sigma_t} \abs{A_{\Sigma_t}} }^2.
\end{split}
\end{equation}
Hence,
\begin{equation}
 \left(\frac{d}{dt} -\Delta_{\Sigma_t} \right) u \leq -\frac{\nabla_{\Sigma_t} v^2}{v^2}\cdot \nabla_{\Sigma_t} u = -2\nabla_{\Sigma_t} \log v \cdot\nabla_{\Sigma_t} u.
\end{equation}

Observe that the hypothesis \eqref{InitialCurvBndEqn} and \eqref{LowSMCFlowEqn} imply that
\begin{equation}
\sup_{\mathcal{S}\cap\partial_P C_{R,t+1}^-(\OO,t)} u \leq c^{-2} \tilde{M}^2 \left(1+R^2+2n (t+1)\right)^{2\alpha}.
\end{equation}
Therefore the standard parabolic maximum principle implies that
\begin{equation}
\sup_{\mathcal{S}\cap C_{R,t+1}^-(\OO,t)} u \leq c^{-2} \tilde{M}^2 \left(1+R^2+2n (t+1)\right)^{2\alpha},
\end{equation}
which proves the result.
\end{proof}

\section{Self-shrinkers of small entropy asymptotic to regular cones}
We define $\mathcal{ACS}_n$ to be the space of connected asymptotically conical self-shrinkers in $\mathbb{R}^{n+1}$. Observe that the hyperplanes through the origin are contained in $\mathcal{ACS}_n$. We denote by $\mathcal{ACS}_n^*$ the subspace of non-flat elements of $\mathcal{ACS}_n$. Furthermore, given a $\lambda\geq 1$, we let $\mathcal{ACS}_n[\lambda]$ be the set of elements of $\mathcal{ACS}_n$ with Gaussian surface area (and hence entropy) less than or equal to $\lambda$ and define $\mathcal{ACS}_n^*[\lambda]$ likewise.

The goal of this section is to prove the long-time existence of a smooth mean curvature flow starting from a perturbation of any element of $\mathcal{ACS}_n^*[\lambda_{n-1}]$. First we need the following properties of the lowest eigenfunction of the stability operator.

\begin{prop}\label{PosEigenValProp}
Given $\Sigma\in\mathcal{ACS}_n^*$, there is a $\mu=\mu (\Sigma)<-1$ and a unique positive smooth function $f$ on $\Sigma$ which satisfies
\begin{equation} \label{EigenEqn}
L_\Sigma f=\Delta_{\Sigma} f -\frac{\xX}{2}\cdot\nabla_{\Sigma} f +\abs{A_{\Sigma}}^2 f+\frac{1}{2} f = -\mu f\quad\mbox{with}\quad \int_\Sigma f^2 e^{-\frac{|\xX|^2}{4}}d\mathcal{H}^n=1.
\end{equation}
Moreover, for each $\beta>0$, there are constants $C_0,\ldots, C_m,\ldots >1$ depending on $\Sigma$ and $\beta$ so that
\begin{equation} \label{EigenBndEqn}
C_0^{-1}\left(1+|\xX (p)|^2\right)^{\frac{1}{2}+\mu-\beta}<f(p)<C_0\left(1+|\xX (p)|^2\right)^{\frac{1}{2}+\mu+\beta}, \quad\mbox{and}
\end{equation}
\begin{equation} \label{EigenDerEqn}
\abs{\nabla^m_\Sigma f(p)}<C_m\left(1+|\xX (p)|^2\right)^{\frac{1}{2}+\mu+\beta-\frac{m}{2}}\quad\mbox{for all $m\geq 1$}.
\end{equation}
\end{prop}

\begin{proof}
Given functions $\phi,\psi$ on $\Sigma$, we define
\begin{equation}
(\phi,\psi)_0=\int_\Sigma\phi\psi e^{-\frac{|\xX|^2}{4}} d\mathcal{H}^n, \quad  (\phi,\psi)_1= (\phi,\psi)_0+(\nabla_\Sigma\phi,\nabla_\Sigma\psi)_0,
\end{equation}
and $\Vert\phi\Vert^2_i=(\phi,\phi)_i$ for $i=0,1$. Let $L_w^2(\Sigma)$ be the Hilbert space of functions on $\Sigma$ with finite $||\cdot ||_0$-norm and $H^1_w(\Sigma)$ be the Hilbert space given by completing $C_c^\infty(\Sigma)$ using $\Vert \cdot \Vert_1$. Observe that $H^1_w(\Sigma)$ may be naturally identified with a subspace of $L^2_w(\Sigma)$. In fact, this embedding is compact -- see Appendix B or \cite[Theorem 3]{ChengZhou}.

Denote by $B^\Sigma_R$ the intrinsic geodesic open ball in $\Sigma$ centered at a fixed point $p\in\Sigma$ with radius $R$. Note that $\bigcup_{R>0} B^\Sigma_R=\Sigma$ as $\Sigma$ is connected. As observed in \cite[Remark 5.17]{CM}, there is a $\mu_R=\mu_R(\Sigma)$ and a unique positive function $f_R\in H_w^1(B^\Sigma_R)$ so that $L_{\Sigma} f_R+\mu_R f_R=0$ with $\Vert f_R \Vert_0^2=1$. Moreover, $\mu_R$ is characterized by
\begin{equation}
\mu_R=\inf_{\phi} \frac{\int_{\Sigma} \left(\abs{\nabla_{\Sigma} \phi}^2 -\abs{A_{\Sigma}}^2 \phi^2 -\frac{1}{2} \phi^2\right) e^{-\frac{|\xX|^2}{4}} d\mathcal{H}^n}{\int_{\Sigma}\phi^2 e^{-\frac{|\xX|^2}{4}} d\mathcal{H}^n},
\end{equation}
where the infimum is taken over all smooth non-zero functions compactly supported in $B^\Sigma_R$. As $\Sigma\in\mathcal{ACS}^*_n$, it follows from \cite[Lemma 2.1]{WaRigid} that
\begin{equation} \label{CurvDecayEqn}
\left(1+|\xX (p)|^2\right) \abs{A_{\Sigma} (p)}^2=O(1).
\end{equation}
Obviously, $\mu_R$ is non-increasing in $R$. Thus, $\mu_R\to\mu$ as $R\to\infty$ for some $\mu> -\infty$. This further implies that $\Vert f_R \Vert_1^2$ is uniformly bounded. Hence, there is a sequence $R_j\to\infty$ so that $f_{R_j}$ converges to a function $f$ on $\Sigma$ weakly in $H_w^1$ and strongly in $L_w^2$. Clearly, $f\in H_w^1(\Sigma)$ satisfying that $L_\Sigma f+\mu f=0$ with $\Vert f \Vert_0^2=1$. The smoothness of $f$ follows from a standard elliptic regularity theory. As each $f_{R_j}$ is non-negative, so is $f$ and, in fact, $f>0$ on $\Sigma$ by the Harnack inequality. Finally, it follows from \cite[Lemma 9.25]{CM} that such an $f$ is unique. 

As $\Sigma$ is non-flat, the classification of mean convex self-shrinkers \cite[Theorem 0.17]{CM} implies that $H_\Sigma$ must change sign. Hence, \cite[Theorem 9.36]{CM} implies that $\mu<-1$. For $\beta>0$, define
\begin{equation}
\bar{g}(\xX)=|\xX|^{1+2\mu+2\beta}\quad\mbox{and}\quad \underline{g}(\xX)=|\xX|^{1+2\mu-2\beta}.
\end{equation}
Then, invoking \cite[Lemma 3.20]{CM} and \eqref{CurvDecayEqn}, a direct computation gives
\begin{equation}
L_\Sigma\bar{g}+\mu\bar{g}<0\quad\mbox{and}\quad L_\Sigma\underline{g}+\mu\underline{g}>0\quad\mbox{on $\Sigma\setminus\bar{B}_{\mathcal{R}}$},
\end{equation}
for some $\mathcal{R}=\mathcal{R}(\beta,\Sigma)>1$ sufficiently large. Next we choose $C>1$ so that
\begin{equation}
C^{-1}\underline{g}<f<C\bar{g}\quad\mbox{on $\partial\left(\Sigma\setminus B_{\mathcal{R}}\right)$}.
\end{equation}
Consider the Dirichlet problem
\begin{equation} \label{DirichletEqn}
\left\{
\begin{array}{ll}
L_\Sigma g+\mu g=0 & \mbox{on $\Sigma\cap\left(B_i\setminus\bar{B}_{\mathcal{R}}\right)$}, \\
g=f & \mbox{on $\partial\left(\Sigma\setminus B_{\mathcal{R}}\right)$}, \\
g=C\bar{g} & \mbox{on $\partial\left(\Sigma\setminus B_i\right)$}.
\end{array}
\right.
\end{equation}
Note that the zeroth order term of the differential equation in \eqref{DirichletEqn} has a negative sign. Thus, it follows from \cite[Theorems 8.3 and 8.13 ]{GT} and the maximum principle that for each $i>\mathcal{R}$, the problem \eqref{DirichletEqn} has a unique smooth solution $g=g_i$ bounded between $C^{-1}\underline{g}$ and $C\bar{g}$. Passing $i\to\infty$, by the Arzela-Ascoli theorem, $g_i$ converges in $C^\infty_{loc} (\Sigma\setminus B_{\mathcal{R}})$ to some function $\tilde{g}$ so that $L_\Sigma\tilde{g}+\mu\tilde{g}=0$ and $\tilde{g}=f$ on $\partial (\Sigma\setminus B_{\mathcal{R}})$. Moreover,
\begin{equation} \label{BoundEqn}
C^{-1} |\xX (p)|^{1+2\mu-2\beta}<\tilde{g} (p)<C |\xX (p)|^{1+2\mu+2\beta}.
\end{equation}
It follows from \eqref{CurvDecayEqn} and elliptic Schauder estimates that $|\nabla_\Sigma\tilde{g}|$ grows at most polynomially. This together with the Euclidean volume growth of $\Sigma$ implies that
\begin{equation} \label{WeightGradEqn}
\int_{\Sigma\setminus B_{\mathcal{R}}}\abs{\nabla_\Sigma\tilde{g}}^2 e^{-\frac{|\xX|^2}{4}} d\mathcal{H}^n<\infty.
\end{equation}

To conclude the proof, it suffices to show that $h=f-\tilde{g}=0$ on $\Sigma\setminus B_{\mathcal{R}}$. Let $\phi_i$ be a cut-off function on $\mathbb{R}^{n+1}$ so that $\phi_i=1$ in $B_i$, $\phi_i=0$ outside $B_{2i}$, and $|D\phi_i|<2/i$. Since $L_\Sigma h+\mu h=0$ with $h=0$ on $\partial (\Sigma\setminus B_{\mathcal{R}})$, multiplying both sides of the equation by $h\phi_i^2 e^{-|\xX|^2/4}$ and integration by parts give that
\begin{equation}
\begin{split}
\int_{\Sigma\setminus B_{\mathcal{R}}} \abs{\nabla_\Sigma h}^2 \phi_i^2 e^{-\frac{|\xX|^2}{4}} d\mathcal{H}^n & -\int_{\Sigma\setminus B_{\mathcal{R}}} \left(\abs{A_\Sigma}^2+\frac{1}{2}+\mu\right) h^2 \phi_i^2 e^{-\frac{|\xX|^2}{4}} d\mathcal{H}^n \\
& +\int_{\Sigma\cap (B_{2i}\setminus B_i)} 2h\phi_i \nabla_\Sigma\phi_i\cdot\nabla_\Sigma h e^{-\frac{|\xX|^2}{4}} d\mathcal{H}^n=0.
\end{split}
\end{equation}
Thus, sending $i\to\infty$ and invoking that $\mu<-1$, the monotone convergence theorem together with \eqref{BoundEqn} and \eqref{WeightGradEqn} implies that $h=0$ on $\Sigma\setminus B_{\mathcal{R}}$ and so \eqref{EigenBndEqn} follows. 

Finally, letting $\Sigma_t=\sqrt{-t}\, \Sigma$, we define 
\begin{equation}
 \tilde{f}(q,t)= (-t)^{\mu+\frac{1}{2}} f \left(\frac{q}{\sqrt{-t}}\right) \quad\mbox{for $t<0$ and $q\in\Sigma_t$}.
\end{equation}
Then \eqref{EigenEqn} implies that
\begin{equation}
\left(\frac{d}{dt}-\Delta_{\Sigma_t}\right)\tilde{f}=\abs{A_{\Sigma_t}}^2\tilde{f}.
\end{equation}
And \eqref{EigenBndEqn} gives that, if $t\in [-1,0)$ and $q\in\Sigma_t\cap (B_2\setminus B_1)$, then
\begin{equation}
0<\tilde{f}(q,t)<C^\prime(\Sigma,\mu,\beta)(-t)^{-\beta}.
\end{equation}
Recall that $\Sigma_t\to\mathcal{C} (\Sigma)$ in $C^\infty_{loc}(\mathbb{R}^{n+1}\setminus\set{\OO})$ as $t\to 0$. Thus, it follows from standard parabolic regularity theory that for each integer $m\geq 1$, there is a $C^{\prime\prime}=C^{\prime\prime} (\Sigma,C^\prime,m)$ so that if $p\in\Sigma\setminus B_R$, then
\begin{equation}
|\nabla_\Sigma^m f(p)|\leq C^{\prime\prime} |\xX (p)|^{1+2\mu+2\beta-m},
\end{equation}
proving \eqref{EigenDerEqn}.
\end{proof}

We next use the lowest eigenfunction of the stability operator to perturb any element of $\mathcal{ACS}_n^*$ in order to strictly decrease its entropy.

\begin{prop}\label{PerturbedInitialSurfaceProp}
Let $\Sigma\in\mathcal{ACS}_n^*$ and let $C_0$, $\mu$, and $f$ be as in Proposition \ref{PosEigenValProp}. There is an $\epsilon_0=\epsilon_0(\Sigma)\in (0,1)$ and  $K=K(\Sigma)>1$ so that if $|\epsilon|<\epsilon_0$, then:
\begin{enumerate} 
 \item \label{PISPi}
 the normal graph over $\Sigma$ given by
 \begin{equation}
 \Gamma^\epsilon=\left\{\yY\in\mathbb{R}^{n+1}: \yY=\xX (p)+\epsilon f(p)\nN_\Sigma (p), \, p\in\Sigma\right\}
 \end{equation} 
 is a smooth hypersurface;
 
 \item \label{PISPv}  
for all $R>1$ we have that $\Gamma^\epsilon \setminus B_{KR} \subset \mathcal{T}_{1/R} \left(\mathcal{C} (\Sigma)\right)$; 
 
 \item  \label{PISPii} 
$\Gamma^\epsilon$ is asymptotically conical and, in particular, given $\delta>0$ there is a $\kappa\in (0,1)$ and $\mathcal{R}>1$ depending only on $\delta$ and $\Sigma$ so that if $p\in\Sigma\setminus B_{\mathcal{R}}$ and $r=\kappa |\xX (p)|$, then $\Gamma^\epsilon\cap B_r(p)$ can be written as a connected graph of a function $w$ over $T_p\Sigma$ with $|Dw|\leq\delta$; 
 
 \item \label{PISPiii} 
 by a suitable choice of the normal to $\Gamma^\epsilon$, 
 \begin{equation} \label{SMCInitialEqn}
S^{O, -1}_{\Gamma^\epsilon}(p)\geq -\mu |\epsilon| C_0^{-1}\left(1+|\xX (p)|^2\right)^{\mu}\geq 0;
 \end{equation}

 \item  \label{PISPiv} 
 if $\epsilon\neq 0$, then $\lambda[\Gamma^\epsilon]<\lambda[\Sigma]$.
\end{enumerate}
\end{prop}

\begin{proof}
By Proposition \ref{PosEigenValProp}, $\mu<-1$, and so if $\beta=\frac{1}{2}\min\set{1, -\mu-1}$, then $0<\beta\leq 1/2$ and $\mu+\beta<-1$. As $\mu$ depends only on $\Sigma$, so does $\beta$. Thus the constants $C_0, C_1$ in Proposition \ref{PosEigenValProp} depend only on $\Sigma$ as well.
Now define the mapping
\begin{equation}
\Xi_\epsilon: \Sigma\to\Gamma^\epsilon\subset\mathbb{R}^{n+1},\quad \Xi_\epsilon (p)=\xX (p)+\epsilon f(p)\nN_\Sigma (p).
\end{equation}
As $\Sigma\in\mathcal{ACS}_n^*$, in view of \cite[Lemma 2.2]{WaRigid}, there is a $\rho>0$ so that $\mathcal{T}_\rho(\Sigma)$ is a regular tubular neighborhood. By Proposition \ref{PosEigenValProp}, $f$ is smooth and uniformly bounded by $C_0$ and so, if $|\epsilon|<\rho/C_0$, then $\Xi_\epsilon(p)=\Psi_\rho(p, \epsilon f(p))$ and so $\Gamma^\epsilon$ is the image under the diffeomorphism $\Psi_\rho$ of the graph of $\epsilon f$ in $\Sigma\times (-\rho,\rho)$. In particular $\Gamma^\epsilon$ is a hypersurface, proving Item \eqref{PISPi}.

As $\mu+\beta<-1$, Item \eqref{PISPv} follows directly from \eqref{EigenBndEqn} for sufficiently large $K=K(\Sigma)$.

Let $\delta>0$ be any given small constant and $\mathcal{R}=\mathcal{R}(\delta,\Sigma)$ be a sufficiently large constant which may change among lines. By \cite[Lemma 2.2]{WaRigid}, there is a $\kappa=\kappa (\delta,\Sigma)\in (0,1)$ so that if $p\in\Sigma\setminus B_{\mathcal{R}}$ and $r=\kappa |\xX (p)|$, then $\Sigma\cap B_{2r} (p)$ is given by a connected graph of a function over $T_p\Sigma$ with gradient bounded by $\delta/2$. As $\beta$ was chosen so $\beta+\mu<-1$,  \eqref{EigenBndEqn} and \eqref{EigenDerEqn} together with \cite[Lemma 2.1]{WaRigid}, and a direct computation (see the proof of \cite[Lemma 2.4]{WaRigid}) imply that if $q\in\Sigma\cap B_{2r} (p)$, then
\begin{equation} \label{NormalOscEqn}
\abs{\nN_{\Gamma^\epsilon} (\Xi_\epsilon (q))\cdot\nN_\Sigma (q)}>1-C\epsilon\left(1+|\xX(p)|\right)^{-2},
\end{equation} 
where $C>0$ depends only on $C_0,C_1$ and $\Sigma$. For $|\epsilon|<\rho/C_0$, $|\Xi_\epsilon (p)-p|<\rho$ and so
\begin{equation} \label{ConnectEqn}
\Gamma^\epsilon\cap B_r (p)\subset \Xi_\epsilon \left(\Sigma\cap B_{2r} (p)\right).
\end{equation}
Combining \eqref{NormalOscEqn} and \eqref{ConnectEqn}, if $p\in\Sigma\setminus B_{\mathcal{R}}$, we conclude that $\Gamma^\epsilon\cap B_r(p)$ can be written as a connected graph of a function over $T_p\Sigma$ with gradient bounded by $\delta$. In view of \eqref{EigenBndEqn} and \eqref{EigenDerEqn}, it follows from Item \eqref{PISPv} that $\Gamma^\epsilon$ is smoothly asymptotic to $\mathcal{C} (\Sigma)$. This completes the proof of Item \eqref{PISPii}.

As computed in the proof of \cite[Lemma 2.4]{WaRigid}, it follows from an appropriate choice of unit normal to $\Gamma^\epsilon$ and \eqref{EigenEqn} that
\begin{equation}
S^{O,-1}_{\Gamma^\epsilon}=2|\epsilon| L_\Sigma f+Q_\Sigma(\epsilon f,\epsilon\nabla_\Sigma f,\epsilon\nabla_\Sigma^2 f)=-2|\epsilon|\mu f+Q_\Sigma(\epsilon f,\epsilon\nabla_\Sigma f,\epsilon\nabla_\Sigma^2 f).
\end{equation}
Where $Q_\Sigma$ is a polynomial all of whose terms are of order at least $2$ and which has uniformly bounded coefficients.  This, together with \eqref{EigenBndEqn}, \eqref{EigenDerEqn} and the fact that $\beta\leq 1/2$, implies Item \eqref{PISPiii} for all $|\epsilon|$ sufficiently small. 

Finally, Item \eqref{PISPiv} can be shown by modifying the arguments in the proof of \cite[Theorem 0.15]{CM}. We leave the technical details to Appendix C.
\end{proof}

We next use sphere barriers to show that smooth mean curvature flows which initially decay to a regular cone must continue to decay to the same cone for all positive time.
 
\begin{lem} \label{NearConeLem}
Let $\set{\Gamma_t}_{t\in [-1,T)}$ be a smooth mean curvature flow in $\Real^{n+1}$. Suppose that there is a regular cone, $\mathcal{C}$, and a $K>1$ so that $\Gamma_{-1}\setminus B_{KR} \subset \mathcal{T}_{1/R} (\mathcal{C})$ for all $R>1$. Then there is a constant $K^\prime=K^\prime (\mathcal{C},K)>1$ so that for all $R>1$ and $t\in [-1,T)$,
\begin{equation}
\Gamma_t \setminus B_{K^\prime\left(R+\sqrt{2n(t+1)}\right)} \subset \mathcal{T}_{K^\prime \left(1+2n(t+1)\right) R^{-1}} (\mathcal{C}).
\end{equation}
\end{lem}
 
\begin{proof}
For $\xX\in\mathbb{R}^{n+1}$, let $\varrho (\xX)=\inf\set{\rho\geq 0: B_\rho (\xX)\cap\left(B_{KR}\cup\mathcal{T}_{1/R} (\mathcal{C})\right)\neq \emptyset}$. Define an open subset of $\Real^{n+1}$ by
\begin{equation}
U=\bigcup_{\xX\in\mathbb{R}^{n+1}} B_{\varrho (\xX)}(\xX).
\end{equation}
Clearly, $U\cap\Gamma_{-1}=\emptyset$. For $t\geq -1$, set
\begin{equation}
\varrho (\xX,t)=\left\{\begin{array}{cc} 
\sqrt{\varrho (\xX)^2-2n(t+1)} & \varrho (\xX)^2\geq 2n(t+1) \\
0 & \varrho (\xX)^2 < 2n(t+1),
\end{array} \right.
\end{equation}
\begin{equation}
U_t=\bigcup_{\xX\in\mathbb{R}^{n+1}} B_{\varrho (\xX,t)}(\xX).
\end{equation} 
By the avoidance principle for mean curvature flow, $U_t\cap\Gamma_t=\emptyset$ for all $t\in [-1,T)$. 

Let $d_{\mathcal{C}}$ denote the distance to the cone $\mathcal{C}$. Observe that 
\begin{equation}
\varrho (\xX)=\max\set{\min\set{|\xX|-KR, d_{\mathcal{C}} (\xX)-R^{-1}},0}
\end{equation}
and so the set $U_t^+=\set{\xX\in\mathbb{R}^{n+1}:\varrho (\xX,t)>0}$ satisfies
\begin{equation}
U_t^+=\set{\xX\in\mathbb{R}^{n+1}: |\xX|>KR+\sqrt{2n(t+1)}, \, d_{\mathcal{C}} (\xX)> R^{-1}+\sqrt{2n(t+1)}}.
\end{equation}
Clearly, $U_t^+\subset U_t$ and, in general,  is a proper subset. Hence,
\begin{equation}
\Gamma_t \setminus B_{KR+\sqrt{2n(t+1)}}\subset\mathcal{T}_{R^{-1}+\sqrt{2n(t+1)}}(\mathcal{C}).
\end{equation}
 	
To conclude, pick a unit normal $\nN_{\mathcal{C}}$ to $\mathcal{C}\setminus\set{\OO}$ so there is a smooth map
\begin{equation}
\Psi : \mathcal{C}\setminus\set{\OO}\times\Real\to\Real^{n+1}, \quad \Psi (p,l)=\xX (p)+l\nN_{\mathcal{C}} (p).
\end{equation}
As $\mathcal{L} (\mathcal{C})$, the link of $\mathcal{C}$, is compact and smooth, there is a $\vartheta\in (0,1/2)$ so that if
\begin{equation}
N=\set{(p,l)\in\mathcal{C}\times\Real: |l|<2\vartheta |\xX(p)|},
\end{equation}
then $\Psi\vert_{N}$ is a diffeomorphism onto its image.

Fix $t > -1$ and consider the part of $U_t^c$ far from the origin. That is, set $K^\prime=1+2K\vartheta^{-1}>4K$ and
\begin{equation}
V_t=U_t^c\cap\set{\xX\in\mathbb{R}^{n+1}: |\xX|\geq K^\prime\left(R+\sqrt{2n(t+1)}\right)}.
\end{equation}
If $\yY\in V_t$, then, by the definition of $U_t$ and the fact that $|\yY|\geq K^\prime R>K R$, $d_\mathcal{C} (\yY)\leq R^{-1}+\sqrt{2n(t+1)}$. Hence, there is a point $p\in\mathcal{C}$ so $d_{\mathcal{C}} (\yY)=|\yY-\xX (p)|\leq R^{-1}+\sqrt{2n (t+1)}$. This implies that 
\begin{equation} \label{LowBndLem43}
\begin{split}
|\xX (p)| &\geq K^\prime\left(R+\sqrt{2n(t+1)}\right)-\sqrt{2n(t+1)}-R^{-1}\\
             & \geq (K^\prime-1)\left(R+\sqrt{2n(t+1)}\right)\\ 
 	       & = 2K \vartheta^{-1} \left(R+\sqrt{2n(t+1)}\right),
\end{split}
\end{equation}
where the second inequality used that $R> 1$ and so $R>R^{-1}$. Hence, 
\begin{equation}
\frac{|\yY-\xX(p)|}{|\xX(p)|}\leq \frac{R^{-1}+\sqrt{2n(t+1)}}{2K \vartheta^{-1} \left(R+\sqrt{2n(t+1)}\right)}< \vartheta, 
\end{equation}
where we used that $R>R^{-1}$ and that $K>1$.
 	
Thus, $\yY\in \Psi(N)$ and so, up to reversing the orientation of $\mathcal{C}$, $\yY=\Psi(p, d_{\mathcal{C}}(\yY))=\xX (p)+d_{\mathcal{C}} (\yY)\nN_{\mathcal{C}} (p)$. Set $\yY_0=\xX(p)+\vartheta |\xX (p)|\nN_{\mathcal{C}} (p)$. Observe that
\begin{equation}
|\yY_0|-\left(\vartheta |\xX (p)|-R^{-1}\right)=\left(\sqrt{1+\vartheta^2}-\vartheta\right)|\xX (p)|+R^{-1}>KR,
\end{equation}
where we used that, by \eqref{LowBndLem43}, $|\xX(p)|>4KR$ and $\sqrt{1+\vartheta^2}-\vartheta>1/3$. Thus, 
\begin{equation}
B_{\vartheta |\xX(p)|-R^{-1}}(\yY_0)\cap B_{KR}=\emptyset \quad\mbox{and so} \quad \varrho(\yY_0)=\vartheta|\xX(p)|-R^{-1}.
\end{equation}
As $\yY\notin U_t$, we have $\yY\notin B_{\rho(\yY_0,t)}(\yY_0)$.  That is, $\rho(\yY_0,t)\leq|\yY-\yY_0|=\vartheta |\xX(p)|-d_{\mathcal{C}}(\yY)$. Hence,
\begin{equation}
\begin{split}
d_\mathcal{C} (\yY) & \leq\vartheta |\xX (p)|-\sqrt{(\vartheta |\xX (p)|-R^{-1})^2-2n(t+1)} \\
                                & = \frac{\vartheta^2 |\xX (p)|^2-(\vartheta |\xX (p)|-R^{-1})^2+2n(t+1)}{\vartheta |\xX (p)|+\sqrt{(\vartheta |\xX (p)|-R^{-1})^2-2n(t+1)}}\\
                                & <\frac{2R^{-1} \vartheta |\xX (p)|-R^{-2}+2n(t+1)}{\vartheta |\xX (p)|} \\
                                & <\frac{2}{R}+\frac{2n(t+1)}{\vartheta |\xX (p)|}<\frac{2(1+2n(t+1))}{\vartheta R}, 
\end{split}
\end{equation}
where we used $|\xX (p)|>R$ for the last inequality. Therefore, as $\Gamma_t\subset U_t^c$ and $K^\prime>2/\vartheta$,
\begin{equation}
\Gamma_t\setminus B_{K^\prime\left(R+\sqrt{2n(t+1)}\right)}\subset\mathcal{T}_{K^\prime \left(1+2n(t+1)\right)R^{-1}} (\mathcal{C}),
\end{equation}
from which the result follows.
\end{proof}

In what follows, we establish the existence and uniqueness of a smooth shrinker mean convex mean curvature flow starting from $\Gamma^\epsilon$ at time $-1$. Moreover, we show that each time slice of this flow is smoothly asymptotic to $\mathcal{C} (\Sigma)$ and if the flow exists smoothly beyond time $0$, then the time zero slice lies strictly on one side of $\mathcal{C} (\Sigma)$.

\begin{prop} \label{PerturbedFlowProp}
Let $\Sigma\in\mathcal{ACS}_n^*$ and let $\Gamma^\epsilon$ be as in Proposition \ref{PerturbedInitialSurfaceProp} for $\epsilon\neq 0$. There is a $T=T(\Gamma^\epsilon)\in (-1,\infty]$, and a unique smooth mean curvature flow $\set{\Gamma_t^\epsilon}_{t\in [-1,T)}$ with $\Gamma_{-1}^\epsilon=\Gamma^\epsilon$. Furthermore, there is an increasing function $R(t)>1$ on $[-1,\infty)$ depending only on $\Gamma^\epsilon$ so that:
\begin{enumerate}
\item \label{PFPi} 
for $t\in [-1,T)$, $\Gamma^\epsilon_t$ is smoothly asymptotic to $\mathcal{C} (\Sigma)$ and, in particular, for $p\in\Gamma^\epsilon_t\setminus B_{R(t)}$, 
\begin{equation}
|\xX (p)|\abs{A_{\Gamma^\epsilon_t} (p)}+|\xX (p)|^2\abs{\nabla_{\Gamma^\epsilon_t} A_{\Gamma^\epsilon_t} (p)}+|\xX (p)|^3\abs{\nabla_{\Gamma^\epsilon_t}^2 A_{\Gamma^\epsilon_t} (p)}<\tilde{C}_0,
\end{equation}
where $\tilde{C}_0$ depends only on $\Gamma^\epsilon$;

\item  \label{PFPii} 
if $T<\infty$, then 
\begin{equation}
\lim_{t\to T}\sup_{B_{R(t)}\cap\Gamma_t^\epsilon} \abs{A_{\Gamma_t^\epsilon}}=\infty;
\end{equation}

\item \label{PFPiv} 
for $t<\min\{0,T\}$, $\Gamma^\epsilon_t\setminus\bar{B}_{R(t)}$ is given by the normal graph of a function $f_t$ over a region $\Omega_t\subset\Sigma_t=\sqrt{-t}\, \Sigma$ with $\Sigma_t\setminus B_{2R(t)}\subset\Omega_t$ which satisfies $f_{-1}=\epsilon f|_{\Omega_{-1}}$, 
\begin{equation}
\tilde{C}_1^{-1} |\epsilon| \left(1+|\xX (p)|\right)^{2\mu}< \abs{f_t(p)}<\tilde{C}_1 \left(1+|\xX (p)|\right)^{-1}, \quad\mbox{and}
\end{equation} 
\begin{equation}
\left(1+|\xX (p)|\right)^2 \abs{\nabla_{\Sigma_t} f_t (p)}+\left(1+|\xX (p)|\right)^3 \abs{\nabla^2_{\Sigma_t} f_t (p)}<\tilde{C}_1,
\end{equation}
where $\tilde{C}_1$ depends only on $\Sigma$;

\item  \label{PFPv}
for $t<\min\{0,T\}$, $\Gamma_t^\epsilon\cap\Sigma_t=\emptyset$, and if $T>0$, then $\Gamma^\epsilon_0\cap\mathcal{C} (\Sigma)=\emptyset$;

\item \label{PFPiii} 
by a suitable choice of the normal to $\Gamma^\epsilon$,
\begin{equation}
S^{O}_{\Gamma^\epsilon_t} (p)\geq -\mu |\epsilon| C_0^{-1} \left(1+|\xX (p)|^2+2n(t+1)\right)^{\mu}>0.
\end{equation}
\end{enumerate}
\end{prop}

\begin{proof}
First, it follows from Items \eqref{PISPv} and \eqref{PISPii} of Proposition \ref{PerturbedInitialSurfaceProp} that $\Gamma^\epsilon$ smoothly asymptotes to the same regular cone, $\mathcal{C} (\Sigma)$, as that of $\Sigma$. Thus, $\Gamma^\epsilon$ has uniformly bounded $C^m$ norm for all $m$ and $\mathcal{T}_\rho(\Gamma^\epsilon)$ is a regular tubular neighborhood of $\Gamma^\epsilon$ for some $\rho>0$. Arguing as in the proof of \cite[Theorem 4.2]{EHInvent}, there exists a smooth mean curvature flow starting from $\Gamma^\epsilon$ for small positive time. Indeed, the construction ensures that each of these time slices can be expressed as a smooth normal graph over $\Gamma^\epsilon$ with small $L^\infty$ norm. Denote the maximal extension of this flow to a smooth mean curvature flow by $\set{\Gamma^\epsilon_t}_{t\in [-1,T)}$.

Invoking Item \eqref{PISPv} of Proposition \ref{PerturbedInitialSurfaceProp}, it follows from Lemma \ref{NearConeLem} and \cite[Lemma 2.2]{WaRigid} that there exist increasing functions $R(t)>1$ and $C(t)>1$ on $[-1,\infty)$ depending only on $\Gamma^\epsilon$ so that for $t\in (-1,T)$ and $p\in\Gamma^\epsilon_t\setminus B_{R(t)}$, $d_\Sigma (\xX (p))\leq C(t)|\xX (p)|^{-1}$, where $d_\Sigma$ is the distance to $\Sigma$. Furthermore, the pseudo-locality property of mean curvature flow (cf. \cite[Theorem 1.5]{IlmanenNevesSchulze}) together with Item \eqref{PISPii} of Proposition \ref{PerturbedInitialSurfaceProp} (for $\delta>0$ sufficiently small) implies that each slice $\Gamma^\epsilon_\tau\cap B_{r} (q)$ is a connected graph of a function $w_\tau$ over $T_q\Sigma$ with $|Dw_\tau|\leq 2\delta$.  Here,  $\tau\leq t$ and
\begin{equation}\label{Chooseqr}
q\in\Sigma \mbox{ is such that } |\xX (p)-\xX (q)|=d_\Sigma (\xX(p))\mbox{ and } r=\varsigma\kappa |\xX (q)|,
\end{equation}  
for some universal small constant $\varsigma$. Replacing the localization function in the proof of \cite[Propositions 3.21 and 3.22]{EckerBook} with $(1-|\xX|^2)^3_+$, we conclude that, for $m\geq 0$,
\begin{equation} \label{CurvDerEqn}
\sup_{\tau\in [-1,t]}\sup_{\Gamma^\epsilon_\tau\cap B_{\frac{r}{2}} (q)} \abs{\nabla^m_{\Gamma^\epsilon_\tau} A_{\Gamma^\epsilon_\tau}} \leq C^\prime (m,\delta,\Gamma^\epsilon) r^{-m-1}.
\end{equation}
This, together with Item \eqref{PISPv} of Proposition \ref{PerturbedInitialSurfaceProp} and Lemma \ref{NearConeLem}, immediately gives Item \eqref{PFPi} and allows us to invoke \cite[Theorem 1.1]{ChenYin} to establish that the flow $\set{\Gamma^\epsilon_t}_{t\in [-1,T)}$ is the unique smooth mean curvature flow starting from $\Gamma^\epsilon$ at time $-1$. Hence, it remains only to verify 
Items \eqref{PFPii}, \eqref{PFPiv}, \eqref{PFPv} and \eqref{PFPiii}.

To show Item \eqref{PFPii}, we argue by contradiction. If $T<\infty$ and
\begin{equation}
\sup_{t\in [-1,T)}\sup_{\Gamma^\epsilon_t} \abs{A_{\Gamma^\epsilon_t}}<\infty,
\end{equation}
then the same holds true for all higher order derivatives of the second fundamental form by \cite[Theorem 3.7]{EHInvent}. As $\Gamma^\epsilon$ decays smoothly to $\mathcal{C} (\Sigma)$, it follows from Item \eqref{PFPi} that $\Gamma^\epsilon_t$ for $t<T$ smoothly decays to $\mathcal{C} (\Sigma)$ in a uniform manner. Thus, $\Gamma^\epsilon_t\to\Gamma^\epsilon_T$ smoothly and $\Gamma^\epsilon_T$ is a hypersurface in $\mathbb{R}^{n+1}$ smoothly decaying to $\mathcal{C}(\Sigma)$. By our previous discussion, we could extend smoothly the flow $\set{\Gamma^\epsilon_t}_{t\in [-1,T)}$ to a larger time interval $[-1,T^\prime)$ for some $T^\prime>T$. This contradicts the maximality of $T$.

By \cite[Lemma 2.2]{WaRigid}, for $\tau\leq t<\min\set{0,T}$, $p\in \Sigma_\tau$ and $q$ and $r$ as in \eqref{Chooseqr},  $\Sigma_\tau\cap B_r(q)$ is the connected graph of a function $\tilde{w}_\tau$ over $T_q\Sigma$ with  $r^{m-1}|D^m\tilde{w}_\tau|$ uniformly bounded for $1\leq m\leq 3$. Below, let $C^{\prime\prime}=C^{\prime\prime} (\delta,\Sigma,\Gamma^\epsilon)$ be a sufficiently large constant which may change among lines. In view of \eqref{EigenBndEqn} and \eqref{EigenDerEqn}, it is straightforward to check that 
\begin{equation} \label{InitialDifferenceEqn}
\sum_{m=0}^2 r^{-1-2\mu-2\beta+m} \norm{D^m w_{-1}-D^m \tilde{w}_{-1}}_{L^\infty} \leq C^{\prime\prime}.
\end{equation}
Thus, combining \eqref{CurvDerEqn} and \eqref{InitialDifferenceEqn}, it follows from the mean curvature flow equation that, as long as $\beta$ is small enough so that $\mu+\beta<-1$, for $t<\min\set{0,T}$,
\begin{equation}
\sum_{m=0}^2 r^{m+1} \norm{D^m w_t-D^m \tilde{w}_t}_{L^\infty} \leq C^{\prime\prime}.
\end{equation}
Hence, using similar reasoning as in the proof of \cite[Lemma 2.3]{WaRigid}, one verifies all of Item \eqref{PFPiv} except the lower bound on $|f_t|$, which will be a consequence of our proof of Item \eqref{PFPv}.

Next, by symmetry, it suffices to show Item \eqref{PFPv} for $\epsilon>0$ and so $f_{-1}=\epsilon f>0$. Let 
\begin{equation}
T_0=\sup\set{t_0<\min\set{0,T}: \Gamma^\epsilon_t\cap\Sigma_t=\emptyset \, \, \mbox{for all $t<t_0$}}.
\end{equation}
We will show that $T_0=\min\set{0,T}$, implying the first part of Item \eqref{PFPv}, i.e., $\Gamma^\epsilon_t\cap\Sigma_t=\emptyset$ when $t<\min\set{0,T}$. First, we prove that $\Gamma^\epsilon_t$ remains disjoint from $\Sigma_t$ for $t$ close to $-1$. By continuity, there is a $\delta^\prime>0$ small so that $\Gamma^\epsilon_t\cap\Sigma_t\cap\bar{B}_{2R(0)}=\emptyset$ for all $t<-1+\delta^\prime$. In particular, if $t<-1+\delta^\prime$, then $f_t$ restricted to $\Sigma_t\cap\partial B_{2R(0)}$ is strictly positive. As both $\Gamma^\epsilon_t$ and $\Sigma_t$ move by mean curvature, the equation for $f_t$ is given by a perturbation of the linearized mean curvature flow equation and thus, invoking \cite[Lemma 2.1]{WaRigid} and Item \eqref{PFPiv}, we get that
\begin{equation} \label{FtEvEqn}
\left(\frac{d}{dt}-\Delta_{\Sigma_t}\right) f_t=\mathbf{a}\cdot\nabla_{\Sigma_t} f_t+b f_t,
\end{equation}
where $\mathbf{a}$ and $b$ are uniformly bounded and depend on $\Sigma$, $f_t$, $\nabla_{\Sigma_t} f_t$, and $\nabla^2_{\Sigma_t} f_t$. The derivation of \eqref{FtEvEqn} involves lengthy but tedious computations for which we refer the reader to \cite[Lemma 2.5]{Sesum}. Define
\begin{equation}
\tilde{f}_t (p)=f_t (p)\eta^{-\mu} (\xX (p),t)=f_t (p)\left(1+|\xX (p)|^2+2n (t+1)\right)^{-\mu}.
\end{equation}
Then it is easy to deduce that
\begin{equation}
\left(\frac{d}{dt}-\Delta_{\Sigma_t}\right) \tilde{f}_t=\tilde{\mathbf{a}}\cdot\nabla_{\Sigma_t} \tilde{f}_t+\tilde{b}\tilde{f}_t,
\end{equation}
where 
\begin{equation}
\tilde{\mathbf{a}}=\mathbf{a}+2\mu\nabla_{\Sigma_t}\log\eta\quad\mbox{and}\quad
\tilde{b}=b+\mu (\mu-1)\abs{\nabla_{\Sigma_t}\log\eta}^2+\mu\mathbf{a}\cdot\nabla_{\Sigma_t} \log\eta.
\end{equation}
Thus, together with \eqref{EigenBndEqn} (choosing $\beta=1/2$), it follows from Theorem \ref{NonCompactMaxPrincipleThm} that, for $t<-1+\delta^\prime$, $\tilde{f}_t$ in $\Sigma_t\setminus B_{2R(0)}$ is bounded from below by a positive constant. In particular, this implies that $T_0\geq -1+\delta^\prime$.

Assume that $T_0<\min\set{0,T}$. Invoking Item \eqref{PFPiv} and \cite[Lemma 2.1]{WaRigid} again, it follows from the parabolic Harnack inequality \cite[Corollary 7.42]{Lieberman} that $\tilde{f}_t$ restricted on $\Sigma_t\setminus B_{4R(0)}$ for $t\leq T_0$ is also bounded from below by a positive constant. Thus, by the parabolic maximum principle on bounded domains, we conclude that $\Gamma^\epsilon_{T_0}\cap\Sigma_{T_0}=\emptyset$. Hence, we repeat the arguments in the previous paragraph to see that $\Gamma^\epsilon_t\cap\Sigma_t=\emptyset$ for $t\in [T_0,T^\prime_0]$ and $T_0^\prime>T_0$. This contradicts the definition of $T_0$. Therefore, $T_0=\min\set{0,T}$, i.e., $\Gamma^\epsilon_t\cap\Sigma_t=\emptyset$ for $t<\min\set{0,T}$.

To conclude, we note that, as $\Sigma_t \to\mathcal{C} (\Sigma)$ in $C^\infty_{loc}(\Real^{n+1}\setminus\set{\OO})$ as $t\to 0$, we must have $\Gamma_0^\epsilon\cap\mathcal{C} (\Sigma)\subset \set{\OO}$ by the previous discussion and parabolic maximum principle on bounded domains. However, if $\Gamma_0^\epsilon\cap\mathcal{C} (\Sigma)=\set{\OO}$, then, as $\set{\Gamma^\epsilon_t}_{t\in [-1,0]}$ is smooth, its parabolic blow-up at $O$ would be a static hyperplane and so  $\Sigma$ would lie in a half-space whose boundary is this hyperplane.   By \cite[Lemma 3.25]{CM} and the fact that $\Sigma$ has polynomial volume growth, this could occur only if $\Sigma$ was the hyperplane -- see also \cite[Theorem 1.1]{CavalcanteEspinar}. This contradicts $\Sigma\in\mathcal{ACS}_n^*$ and shows Item \eqref{PFPv}. Finally, Item \eqref{PFPiii} follows from Proposition \ref{LowBndSMCProp}, Item \eqref{PFPi} and Item \eqref{PISPiii} of Proposition \ref{PerturbedInitialSurfaceProp}.
\end{proof}

 Following the arguments in \cite{CIMW}, we show that, for elements of $\mathcal{ACS}^*_n$ with small entropy, the flow constructed in Proposition \ref{PerturbedFlowProp} exists for long time.  A crucial ingredient in our proof is the fact that shrinker mean convexity is preserved under mean curvature flow -- i.e., Item \eqref{PFPiii} of Proposition \ref{PerturbedFlowProp}.

\begin{prop} \label{SingularProp}
If $\Sigma\in\mathcal{ACS}^*_n[\lambda_{n-1}]$ and, for $\epsilon\neq 0$, $\set{\Gamma^\epsilon_t}_{t\in [-1,T)}$ is given by Proposition \ref{PerturbedFlowProp}, then $T=\infty$.
\end{prop}

\begin{proof}
We argue by contradiction. If $T<\infty$, then Item \eqref{PISPiv} of Proposition \ref{PerturbedInitialSurfaceProp}, Item \eqref{PFPii} of Proposition \ref{PerturbedFlowProp} and Brakke's local regularity theorem \cite{B} or \cite{WhiteReg} imply that there is an $\xX_0\in \bar{B}_{R(T)}$ so that the corresponding rescaled flow about $X_0=(\xX_0,T)$, 
\begin{equation}
\tilde{\Gamma}^\epsilon_s=(T-t)^{-\frac{1}{2}}\left(\Gamma^\epsilon_t-\xX_0\right),\quad s=-\log (T-t),
\end{equation}
satisfies that for some sequence $s_i\to\infty$, $\tilde{\Gamma}^\epsilon_{s_i}$ converges to a multiplicity one $F$-stationary varifold $\tilde{\Sigma}$ with $1<\lambda [\tilde{\Sigma}]<\lambda_{n-1}$ and
\begin{equation} \label{L2ShrinkerEqn}
\int_{\tilde{\Gamma}^\epsilon_{s_i}} \abs{\mathbf{H}_{\tilde{\Gamma}^\epsilon_{s_i}}+\frac{\xX^\perp}{2}}^2 e^{-\frac{|\xX|^2}{4}} d\mathcal{H}^n\to 0.
\end{equation}

Invoking Items \eqref{PISPii} and \eqref{PISPiii} of Proposition \ref{PerturbedInitialSurfaceProp} and Item \eqref{PFPi} of Proposition \ref{PerturbedFlowProp}, it follows from Proposition \ref{CurvBndSMCProp} that on $\tilde{\Gamma}^\epsilon_{s_i}\cap B_{R_i}$ for $R_i=R(T) e^{s_i/2}$,
\begin{equation} \label{CurvBndEqn}
\abs{A_{\tilde{\Gamma}^\epsilon_{s_i}}}<C(\Sigma,\Gamma^\epsilon,T)\left\{2\left(e^{-s_i}-T\right)H_{\tilde{\Gamma}^\epsilon_{s_i}}-e^{-\frac{s_i}{2}}\left(\xX_0+e^{-\frac{s_i}{2}}\xX\right)\cdot\nN_{\tilde{\Gamma}^\epsilon_{s_i}}\right\}.
\end{equation}
Passing $s_i\to\infty$, this together with Brakke's local regularity theorem implies that 
\begin{equation} \label{CurvRegEqn}
\abs{A_{\tilde{\Sigma}}} \leq -2CTH_{\tilde{\Sigma}}\quad\mbox{on the regular set $\mbox{Reg}(\tilde{\Sigma})$}.
\end{equation}
If $n\geq 3$, then $\lambda_{n-1}<3/2$ and so $\tilde{\Sigma}$ is smoothly embedded by \cite[Proposition 5.1]{CIMW}. If $n=2$, it follows from \eqref{L2ShrinkerEqn} and \eqref{CurvBndEqn} that
\begin{equation}
\int_{\tilde{\Gamma}^\epsilon_{s_i}\cap B_R} \abs{A_{\tilde{\Gamma}^\epsilon_{s_i}}}^2 e^{-\frac{|\xX|^2}{4}} d\mathcal{H}^2 < C^\prime (\lambda_1,R,T,C).
\end{equation}
Thus, in view of \cite[Lemma 4]{SimonWillmore}, the singular set $\mbox{Sing}(\tilde{\Sigma})$ is discrete. Furthermore, as $\tilde{\Sigma}$ is of multiplicity one, it follows from Allard's regularity theorem that $\mbox{Sing}(\tilde{\Sigma})=\emptyset$.

If $T=0$, then $\tilde{\Sigma}$ is flat and of multiplicity one and hence, $\lambda [\tilde{\Sigma}]=1$, contradicting $\lambda [\tilde{\Sigma}]>1$. If $T\neq 0$, then $H_{\tilde{\Sigma}}$ does not change sign and, as $F[\tilde{\Sigma}]<\lambda_{n-1}$, it follows from \cite[Theorem 0.14]{CM} that $\tilde{\Sigma}$ is the self-shrinking sphere or a hyperplane. As $\tilde{\Gamma}^\epsilon_{s_i}$ is connected and non-compact, $\tilde{\Sigma}$ must be also non-compact and so $\tilde{\Sigma}$ cannot be a sphere. Hence, $\tilde{\Sigma}$ must be a hyperplane, yielding a contradiction as in the previous case.
\end{proof}

We will need the following straightforward consequence of Propositions \ref{PerturbedFlowProp} and \ref{SingularProp} in the next section.
\begin{thm} \label{StarshapThm}
Let $\Sigma\in\mathcal{ACS}^*_n[\lambda_{n-1}]$ and let $\Gamma^\epsilon$ be as in Proposition \ref{PerturbedInitialSurfaceProp} for $\epsilon\neq 0$. There is a unique smooth mean curvature flow $\set{\Gamma^\epsilon_t}_{t\in [-1,0]}$ with $\Gamma^\epsilon_{-1}=\Gamma^\epsilon$. Moreover, $\Gamma^\epsilon_0$ is a connected hypersurface which is smoothly asymptotic to, but disjoint from, $\mathcal{C}(\Sigma)$, and which satisfies $\xX\cdot\nN_{\Gamma^\epsilon_0}>0$. 
\end{thm}

\section{Topology of Non-Compact Shrinkers of Small Entropy} \label{FinalSec}
We first observe the following topological fact which essentially says that the hypersurfaces $\Gamma^\epsilon_0$ from Theorem \ref{StarshapThm} are star-shaped relative to $\OO$.

\begin{prop} \label{RadialGraphTopProp}
Let $\Sigma\subset \Real^{n+1}$ be a connected hypersurface which is smoothly asymptotic to a regular cone $\mathcal{C}(\Sigma)$. If $\xX\cdot \nN_{\Sigma}>0$ and $\Sigma\cap \mathcal{C}(\Sigma)=\emptyset$, then the map $\Pi:\Sigma\to \mathbb{S}^n$ given by
\begin{equation}
\Pi(p)=\frac{\xX(p)}{|\xX(p)|}
\end{equation}
is a diffeomorphism onto its image. Furthermore, the image is a connected component of $\mathbb{S}^n\backslash \mathcal{L}(\Sigma)$, where $\mathcal{L}(\Sigma)$ is the link of $\mathcal{C}(\Sigma)$.
\end{prop}

\begin{proof}
Let $\Omega=\Pi(\Sigma)$ and observe that, as $\Pi(\mathcal{C}(\Sigma))=\mathcal{L}(\Sigma)$ and $\Sigma\cap \mathcal{C}(\Sigma)=\emptyset$, $\Omega\cap \mathcal{L}(\Sigma)=\emptyset$. It follows that $\Pi$ is a proper map. Indeed, suppose that $\mathcal{K}\subset\Omega$ is compact. As $\mathcal{K}\cap\mathcal{L}(\Sigma)=\emptyset$ and $\Sigma$ is smoothly asymptotic to $\mathcal{C}(\Sigma)$, there is an $R=R(\mathcal{K})>0$ so that $\Pi^{-1}(\mathcal{K})\subset B_{R}$.  As $\Sigma$ is a hypersurface, and so proper, $\Sigma\cap\bar{B}_R$ is compact and so $\Pi^{-1}(\mathcal{K})\cap B_{R}$ is also compact.

Next notice that, as $\xX\cdot\nN_{\Sigma}>0$ and
\begin{equation}
\mathrm{d}\Pi_p(\vV)=\frac{\vV}{|\xX(p)|}-\frac{\xX(p) \left(\xX(p)\cdot\vV\right)}{|\xX(p)|^3} \quad\mbox{for $\vV\in T_p\Sigma$},
\end{equation}
$\mathrm{d}\Pi_p$ is bijective for all $p\in\Sigma$. Hence, $\Pi$ is a local diffeomorphism and so $\Omega$ is an open subset of $\mathbb{S}^n$.   As $\Omega$ is the image of a connected set, $\Omega$ itself must be connected and hence, by a standard topological result, $\Pi$ is a finite covering map.

We now show that $\Pi$ is a diffeomorphism. To see this, fix a $q\in \Omega$ and suppose that $p,p^\prime\in\Pi^{-1}(q)$ are two consecutive points in the pre-image of $q$ -- here we order by distance from $\OO$. For $\delta>0$ sufficiently small, let $\tilde{p}=\xX(p)+\delta\xX(q)$ and likewise define $\tilde{p}^\prime$. Invoking that $\xX\cdot\nN_{\Sigma}>0$, the straight line segment $\overline{\tilde{p}\tilde{p}^\prime}$ connecting $\tilde{p}$ to $\tilde{p}^\prime$ intersects transversally with $\Sigma$ exactly once. Now let $\Pi_\Sigma$ be the nearest point projection onto $\Sigma$ and let $\gamma: [0,1]\to\Sigma$ be a simple, smooth parametrized curve connecting $\Pi_\Sigma(\tilde{p})$ to $\Pi_\Sigma(\tilde{p}^\prime)$. Consider the curve $\tilde{\gamma}: [0,1]\to\mathbb{R}^{n+1}$ defined by 
\begin{equation}
\tilde{\gamma}(t)=\xX (\gamma(t))+\delta\nN_{\Sigma} (\gamma(t)).
\end{equation}
Clearly, $\overline{\tilde{p}\tilde{\gamma}(0)}\cup\tilde{\gamma}\cup\overline{\tilde{\gamma}(1)\tilde{p}^\prime}$ is homotopic to $\overline{\tilde{p}\tilde{p}^\prime}$ in $\mathbb{R}^{n+1}$, but does not intersect $\Sigma$. This is a contradiction and completes the proof that $\Pi$ is a diffeomorphism of $\Sigma$ onto $\Omega$.

\begin{figure}
 \centering
\resizebox{4in}{!}{\input{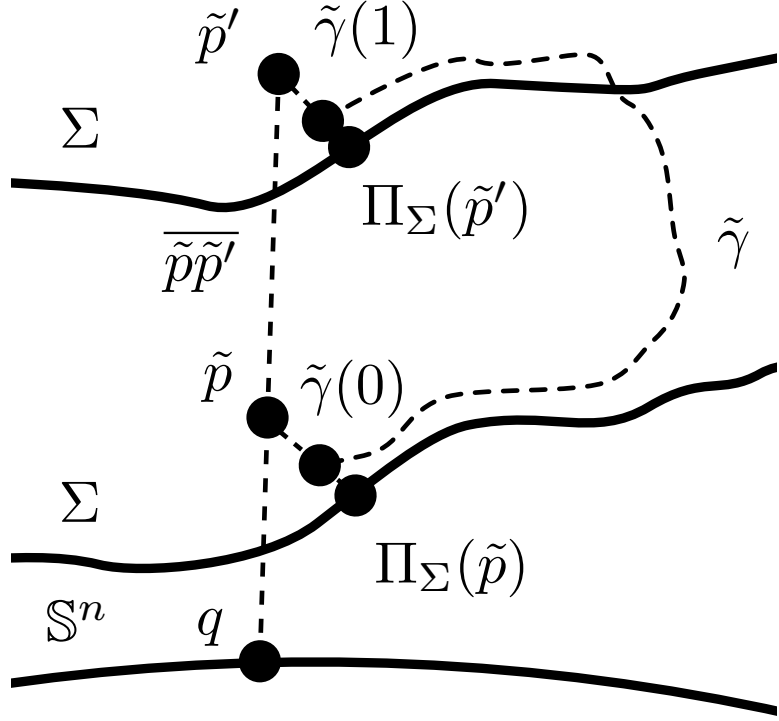}}
\caption{The line segment $\overline{\tilde{p}\tilde{p}'}$ and curve $\tilde{\gamma}$. }
\label{Sec5Img}
\end{figure}

Finally, we show that $\Omega$ must be a connected component of $\mathbb{S}^n\setminus\mathcal{L}(\Sigma)$. To see this, let $\hat{\Omega}$ be the connected component of $\Real^{n+1}\setminus\mathcal{C}(\Sigma)$ which contains $\Sigma$ -- such a component exists as $\Sigma$ is connected and disjoint from $\mathcal{C}(\Sigma)$. Clearly, $\Omega\subset\hat{\Omega}\cap \mathbb{S}^{n}$ and $\hat{\Omega}\cap \mathbb{S}^n$ is a connected component of $\mathbb{S}^n\setminus\mathcal{L}(\Sigma)$. Pick a point $q\in\partial\Omega$ and a sequence $q_i\in \Omega$ with $q_i \to q\not\in\Omega$.  Set $p_i=\Pi^{-1}(q_i)\in\Sigma$. As $\Sigma$ is proper, $|\xX(p_i)|\to \infty$ and so, as $\Sigma$ is smoothly asymptotic to $\mathcal{C}(\Sigma)$, $q\in \mathcal{L}(\Sigma)$.  Hence, $\partial \Omega\subset \mathcal{L}(\Sigma)$ and our claim is verified.
 \end{proof}

We may now prove Theorem \ref{MainACSThm}.

\begin{proof}[Proof of Theorem \ref{MainACSThm}]
First observe that the theorem is trivially true for hyperplanes.   We next claim that $\Sigma$ must be connected. Indeed,  otherwise one could use two of the connected components as barriers to find a stable self-shrinker which is impossible in view of \cite[Lemma 3.3]{BernsteinWangRemark}. As such, we may assume that $\Sigma\in \mathcal{ACS}_n^*[\lambda_{n-1}]$.
 
Hence, applying Theorem \ref{StarshapThm}, we associate to $\Sigma$ two smooth mean curvature flows $\set{\Gamma_t^{\pm \epsilon}}_{t\in[-1,0]}$ with $\Gamma_{-1}^{\pm\epsilon}=\Gamma^{\pm\epsilon}$ for some $\epsilon>0$ small, where $\Gamma^{\pm\epsilon}$ are the hypersurfaces as in Proposition \ref{PerturbedInitialSurfaceProp}. Moreover, each $\Gamma^{\pm}=\Gamma^{\pm\epsilon}_0$ is a connected star-shaped hypersurface which is smoothly asymptotic to, but disjoint from, $\mathcal{C}(\Sigma)$. Notice that $\Real^{n+1}\backslash \Sigma$ consists of exactly two connected components (as $\Sigma$ is connected) one containing $\Gamma^{-\epsilon}$ and the other $\Gamma^\epsilon$. Hence, $\Real^{n+1}\setminus\mathcal{C}(\Sigma)$ also consists of exactly two connected components $\hat{\Omega}^\pm$ each containing $\Gamma^\pm$ and so, by Proposition \ref{RadialGraphTopProp}, each $\Gamma^\pm$ is diffeomorphic to $\Omega^\pm=\hat{\Omega}^\pm\cap \mathbb{S}^n$.  This proves all but the last claim.

The final claim, that $\mathcal{L}(\Sigma)$ is connected, follows by applying the Mayer-Vietoris long exact sequence for reduced homology.  Indeed, slightly enlarge $\Omega^\pm$ so that both are still connected and $\Omega^+\cap \Omega^-$ is a regular tubular neighborhood of $\mathcal{L}(\Sigma)$ in $\mathbb{S}^n$. One then has
\begin{equation}
\tilde{H}_1(\mathbb{S}^n)\to\tilde{H}_0(\mathcal{L}(\Sigma))\to\tilde{H}_0 (\Omega^-)\oplus\tilde{H}_0 (\Omega^+)
\end{equation}
as part of the long exact sequence. As $n\geq 2$ and both $\Omega^-$ and $\Omega^+$ are connected, 
\begin{equation}
\set{0}=\tilde{H}_1(\mathbb{S}^n)=\tilde{H}_0(\Omega^-)=\tilde{H}_0(\Omega^+).
\end{equation}
Hence,  $\tilde{H}_0(\mathcal{L}(\Sigma))=\set{0}$ and so $\mathcal{L}(\Sigma)$ is connected.
\end{proof}

We next prove Corollary \ref{ShrinkerClassificationCor}.

\begin{proof}[Proof of Corollary \ref{ShrinkerClassificationCor}]
First, notice that all possible singularities at the first singular time of mean curvature flow of closed surfaces have smooth supports -- see \cite[Theorem 2]{I}. Hence, one can strengthen \cite[Theorem 0.8]{CIMW} of Colding-Ilmanen-Minicozzi-White to conclude that any closed self-shrinker with entropy (equivalently, Gaussian surface area) less than or equal to $\lambda_1$ is diffeomorphic to a sphere. Furthermore, such a self-shrinker must be $\mathbb{S}^2_*$ by a recent result of Brendle \cite{Brendle} on the uniqueness of closed genus zero self-shrinkers (recall, in our definition self-shrinkers are automatically embedded).

Next, as observed in \cite[Proposition 4.7]{BernsteinWang} and in the proof of \cite[Theorem 1.2]{BernsteinWang} when $n=2$, any non-compact self-shrinker in $\Real^3$ with $\lambda[\Sigma]\leq \lambda_1$ is either a rotation of $\mathbb{S}^1_*\times \Real$ or asymptotically conical. By Theorem \ref{MainACSThm}, in the latter case, $\Sigma$ is diffeomorphic to a open disk. Invoking \cite[Theorem 2]{Brendle} of Brendle, it follows that $\Sigma$ must be flat. Thus, the only self-shrinkers with entropy less than or equal to $\lambda_1$ are $\mathbb{S}^2_*$ and rotations of $\Real^2\times \set{0}$ and $\mathbb{S}_*^1\times\Real$.

To prove the gap, we argue by contradiction. Namely, suppose that there was no such $\delta_0$. Then there would exist a sequence $\Sigma_i$ of self-shrinkers with $\lambda[\Sigma_i]\in (\lambda_1,2)$ and $\lambda[\Sigma_i]\to\lambda_1$.  Up to passing to a subsequence, we may assume that $\Sigma_i$ converges to a multiplicity one $F$-stationary varifold $\Sigma$ with $\lambda [\Sigma]\leq\lambda_1$. As each $\Sigma_i$ is two-sided, the tangent cones of $\Sigma$ cannot be quasi-planes. Also, there is no minimal cone with isolated singularities in $\mathbb{R}^3$. Thus, by Allard's regularity theorem, $\Sigma_i$ converges in $C^\infty_{loc}(\Real^3)$ to $\Sigma$ and $\lambda[\Sigma]=\lim_{i\to \infty} \lambda[\Sigma_i]=\lambda_1$. By the rigidity of self-shrinkers of small entropy established in the previous paragraph, this implies that (up to a rotation) $\Sigma=\mathbb{S}_*^1\times \Real$. Hence, in view of \cite[Theorem 0.1]{CIM}, for all $i$ sufficiently large, $\Sigma_i$ are rotations of $\mathbb{S}^1_*\times\mathbb{R}$ and, in particular, $\lambda[\Sigma_i]=\lambda_1$, giving the desired contradiction.
\end{proof}

Finally we show Corollary \ref{TopSurfaceCor}.

\begin{proof}[Proof of Corollary \ref{TopSurfaceCor}]
Let $\set{\Sigma_t}_{t\in [-1,T)}$ be the maximal smooth mean curvature flow with initial surface $\Sigma_{-1}=\Sigma$. As $\Sigma$ is closed, a comparison with a large shrinking sphere implies that $T<\infty$. Let $X_0=(\xX_0,T)$ be a singular point of the flow. It follows from Corollary \ref{ShrinkerClassificationCor} and the monotonicity of entropy that $\lambda[\Sigma]\geq\lambda_2$ with equality if and only if, modulo translations and scalings, $\Sigma$ is equal to $\mathbb{S}^2$. If the tangent flow at $X_0$ is a self-shrinking sphere, then, by Brakke's local regularity theorem, $\Sigma_t$ must be very close to the sphere for all $t$ near $T$. In particular, $\Sigma_t$ is diffeomorphic to $\mathbb{S}^2$ and thus, as the flow is smooth, so is $\Sigma$. Hence, if $\Sigma$ has positive genus, the tangent flow cannot be a shrinking sphere and so $\lambda[\Sigma]>\lambda_1$ by the monotonicity formula and Corollary \ref{ShrinkerClassificationCor}.
\end{proof}

\appendix

\section{}
We use a weighted version of Huisken's monotonicity formula to establish a parabolic maximum principle for smooth non-compact mean curvature flows with boundaries. This is a slight generalization of \cite[Corollary 1.1]{EH}.

\begin{thm} \label{NonCompactMaxPrincipleThm}
Let $\mathcal{S}=\set{\Sigma_t}_{t\in [-1,T)}$ be a smooth mean curvature flow in $\mathbb{R}^{n+1}$ with finite entropy. For a fixed $R\geq 0$ and $X_0=(\OO,-1)$, suppose that $u$ is a $C^2$ function on $\mathcal{S}\setminus C_{R,T+1} (X_0)$ which satisfies:
\begin{enumerate}
\item \label{NCMPi} 
on $\mathring{\mathcal{S}}=\mathcal{S}\setminus\overline{C_{R,T+1} (X_0)}$,
\begin{equation}
\left(\frac{d}{dt}-\Delta_{\Sigma_t}\right) u \geq\mathbf{a}\cdot \nabla_{\Sigma_t} u+ bu
\end{equation}
with $\sup_{\mathring{\mathcal{S}}} |\mathbf{a}|^2+|b|=M_0<\infty$; 

\item \label{NCMPii}
$u>c_0\geq 0$ on $\partial_P\mathring{\mathcal{S}}=\left(\Sigma_{-1}\setminus B_R\right)\cup\left(\mathcal{S}\cap\partial C_{R,T+1} (X_0)\right)$ for some constant $c_0$;

\item \label{NCMPiii} for all $t\in [-1,T)$,
\begin{equation}
\int_{\Sigma_t\setminus B_R} \left(\abs{u}^2+\abs{\frac{du}{dt}}^2+\abs{\nabla_{\Sigma_t} u}^2+\abs{\nabla_{\Sigma_t}^2 u}^2\right) \Phi_{(\OO,T)} \, d\mathcal{H}^n<\infty,
\end{equation}
where
\begin{equation}
\Phi_{(\OO,T)} (\xX,t)=(T-t)^{-\frac{n}{2}} \Phi\left(\frac{\xX}{\sqrt{T-t}}\right)=\left(4\pi (T-t)\right)^{-\frac{n}{2}}e^{-\frac{|\xX|^2}{4(T-t)}}.
\end{equation}
\end{enumerate}
Then for all $t\in [-1,T)$,
\begin{equation}
\inf_{\Sigma_t \setminus B_R} u \geq c_0 e^{M_1 (t+1)} \geq 0,
\end{equation}
where $M_1=\min\set{\inf_{\mathring{\mathcal{S}}} b,0} \leq 0$.
\end{thm}

\begin{proof}	
We will first show that $u\geq 0$. Let $u_0(p,t)=\max\set{-u(p,t),0}\geq 0$ on $\mathring{\mathcal{S}}$ and we extend $u_0$ by zero to the whole $\mathcal{S}$. As $u$ is continuous and, for every $0<T^\prime<T$, $\mathcal{S}\cap\partial_P C_{R,T^\prime+1} (X_0)$ is compact, Item \eqref{NCMPii} implies that $-u$ is strictly negative in a neighborhood of $\mathcal{S}\cap\partial_P C_{R,T^\prime+1} (X_0)$. Hence, $u^2_0$ is $C^{1,1}$ and satisfies, in the weak sense, that
\begin{equation} \label{u0EvEqn}
\left(\frac{d}{dt}-\Delta_{\Sigma_t}\right) u_0^2 \leq -2\abs{\nabla_{\Sigma_t} u_0}^2+2u_0 \mathbf{a}\cdot\nabla_{\Sigma_t} u_0+2b u_0^2\leq 3M_0 u_0^2.
\end{equation}
Here we used Item \eqref{NCMPi} and Young's inequality.

Invoking \eqref{u0EvEqn} and Item \eqref{NCMPiii}, it follows from the weighted monotonicity formula \cite[Theorem 4.13]{EckerBook} that
\begin{equation}
\frac{d}{dt} \int_{\Sigma_t} u_0^2 \, \Phi_{(\OO,T)} \, d\mathcal{H}^n \leq 3M_0 \int_{\Sigma_t} u_0^2 \, \Phi_{(\OO,T)} \, d\mathcal{H}^n.
\end{equation}
This together with Item \eqref{NCMPii} implies that for all $t\in [-1,T)$,
\begin{equation}
\int_{\Sigma_t} u_0^2 \, \Phi_{(\OO,T)} \, d\mathcal{H}^n=0.
\end{equation}
That is, $u_0\equiv 0$ on $\mathcal{S}$ and so $u\geq 0$ as claimed.

Set $\tilde{u}=e^{-M_1 (t+1)} u$. Using $\tilde{u}\geq 0$ and $b-M_1\geq 0$, it follows that
\begin{equation}
\left(\frac{d}{dt}-\Delta_{\Sigma_t}\right) \tilde{u} \geq \mathbf{a}\cdot \nabla_{\Sigma_t} \tilde{u}+\left(b-M_1\right)\tilde{u} \geq \mathbf{a}\cdot\nabla_{\Sigma_t} \tilde{u}. 
\end{equation} 
As $M_1 \leq 0$, $\tilde{u}>c_0$ on $\partial_P\mathring{\mathcal{S}}$. Applying the same argument as above to $\tilde{u}-c_0$, we conclude that $\tilde{u} \geq c_0$ which completes the proof.
\end{proof}

\section{}
Throughout this appendix, let $\Sigma$ be a self-shrinker in $\mathbb{R}^{n+1}$ and let $H^1_w(\Sigma)$ and $L^2_w(\Sigma)$ be as in the proof of Proposition \ref{PosEigenValProp}. The goal is to prove the natural embedding 
\begin{equation}
\iota : H^1_w(\Sigma) \hookrightarrow L^2_w(\Sigma)
\end{equation}
is compact.

First, we prove a lemma that yields a uniform decay rate for the weighted $L^2$ integration over exterior regions.

\begin{lem} \label{ExtWeightL2Lem}
For all $\phi\in C^\infty_c (\Sigma)$, 
\begin{equation}
\int_\Sigma \phi^2 |\xX|^2 e^{-\frac{|\xX|^2}{4}} d\mathcal{H}^n \leq 16 \int_\Sigma |\nabla_\Sigma\phi|^2 e^{-\frac{|\xX|^2}{4}} d\mathcal{H}^n+4n \int_\Sigma \phi^2 e^{-\frac{|\xX|^2}{4}} d\mathcal{H}^n
\end{equation}
\end{lem}

\begin{proof}
Recall that the divergence theorem implies that
\begin{equation}
\int_\Sigma{\rm div}_\Sigma \mathbf{V} d\mathcal{H}^n=-\int_\Sigma\mathbf{H}_\Sigma\cdot \mathbf{V} d\mathcal{H}^n
\end{equation}
for all vector fields $\mathbf{V}$ of compact support. Setting $\mathbf{V}=\xX\phi^2 e^{-|\xX|^2/4}$ yields the lemma. 

Indeed, we compute
\begin{equation}
{\rm div}_\Sigma \mathbf{V}=n\phi^2 e^{-\frac{|\xX|^2}{4}}+2(\xX\cdot\nabla_\Sigma\phi)\phi e^{-\frac{|\xX|^2}{4}}-\frac{1}{2}\phi^2 |\xX^\top|^2 e^{-\frac{|\xX|^2}{4}},
\end{equation}
and, using the self-shrinker equation and Young's inequality, 
\begin{equation}
\begin{split}
\frac{1}{2}\int_\Sigma \phi^2 |\xX|^2 e^{-\frac{|\xX|^2}{4}} d\mathcal{H}^n & =n\int_\Sigma \phi^2 e^{-\frac{|\xX|^2}{4}} d\mathcal{H}^n+2\int_\Sigma \left(\xX\cdot\nabla_\Sigma\phi\right) \phi e^{-\frac{|\xX|^2}{4}} d\mathcal{H}^n \\
& \leq n\int_\Sigma \phi^2 e^{-\frac{|\xX|^2}{4}} d\mathcal{H}^n+4\int_\Sigma \abs{\nabla_\Sigma\phi}^2 e^{-\frac{|\xX|^2}{4}} d\mathcal{H}^n \\
& +\frac{1}{4}\int_\Sigma \phi^2 |\xX|^2 e^{-\frac{|\xX|^2}{4}} d\mathcal{H}^n.
\end{split}
\end{equation}
The desired estimate follows from a rearrangement of the above inequality.
\end{proof} 

Now we are ready to prove that

\begin{prop}
The natural embedding $\iota : H_w^1(\Sigma) \hookrightarrow L^2_w(\Sigma)$ is compact.
\end{prop}

\begin{proof}
As $C^\infty_c (\Sigma)$ is dense in $H_w^1(\Sigma)$, it suffices to prove the proposition for $\iota$ restricted to $C^\infty_c (\Sigma)\cap H^1_w(\Sigma)$. Namely, let $\phi_i$ be a sequence of elements of $C^\infty_c(\Sigma)\cap H^1_w(\Sigma)$ with $(\phi_i,\phi_i)_1\leq C$ for some constant $C>0$. We need to find a subsequence $\phi_{i_j}$ converging strongly in $L^2_w(\Sigma)$.

For $R>1$ we introduce the cut-off function $\psi_R: \mathbb{R}^{n+1}\to [0,1]$ so that $\psi_R=1$ in $B_R$, $\psi_R=0$ outside $B_{2R}$, and $|D\psi_R|\leq 2/R$. Let $\phi^R_i=\phi_i\psi_R$. Clearly ${\rm spt} (\phi_i^R)\subset\Sigma\cap B_{2R}$ and $(\phi_i^R,\phi_i^R)_1\leq 8C$ for all $i$. By the Rellich-Kondrachov compactness theorem, there is a subsequence of $\phi_i^R$ converging strongly in $L^2_w(\Sigma)$ to $\phi^R$ supported in $\Sigma\cap B_{2R}$. Now take an increasing sequence $R_k\to\infty$ and by a diagonalization argument we find a subsequence $i_j\to\infty$ and $\phi\in L^2_w(\Sigma)$ so that $\phi_{i_j}\to\phi$ strongly in $L^2_w(\Sigma\cap B_{R_k})$ for each $k$. 

Observe that $\Vert\phi^{R_k}-\phi\Vert_0\to 0$ as $k\to\infty$. And by Lemma \ref{ExtWeightL2Lem}, $\Vert\phi_{i_j}-\phi_{i_j}^{R_k}\Vert_0\leq C^\prime R_k^{-2}$ for some $C^\prime$ depending only on $n$ and $C$. Thus
\begin{equation}
\Vert\phi_{i_j}-\phi\Vert_0\leq\Vert\phi_{i_j}-\phi_{i_j}^{R_k}\Vert_0+\Vert\phi_{i_j}^{R_k}-\phi^{R_k}\Vert_0+\Vert\phi^{R_k}-\phi\Vert_0\to 0,
\end{equation}
i.e., $\phi_{i_j}\to\phi$ strongly in $L^2_w(\Sigma)$.
\end{proof}

\section{}
Let $\Sigma\in\mathcal{ACS}^*_n$ and let $f$ be the lowest eigenfunction of $L_\Sigma$ as in Proposition \ref{PosEigenValProp}. Throughout $\epsilon_1>0$ is assumed to be sufficiently small. We consider a one-parameter family of hypersurfaces 
\begin{equation}
\Gamma^s=\set{\yY\in\mathbb{R}^{n+1}: \yY=\xX(p)+sf(p)\nN_\Sigma (p), p\in\Sigma}
\end{equation}
for $|s|\leq\epsilon_1$ with $\Gamma^0=\Sigma$. Observe that, as $\Gamma^s$ is asymptotically conical by Item \eqref{PISPii} of Proposition \ref{PerturbedInitialSurfaceProp}, for each $|s|\leq\epsilon_1$, this family of hypersurfaces is a normal variation of $\Gamma^s$ with vector field $f_{\Gamma^s} \nN_{\Gamma^s}$, where $f_{\Gamma^s}=f\nN_\Sigma\cdot\nN_{\Gamma^s}$. Define a function 
\begin{equation}
G: \mathbb{R}^{n+1}\times\mathbb{R}^+\times [-\epsilon_1,\epsilon_1]\to\mathbb{R}^+,\quad G(\xX_0,t_0,s)=\left(4\pi t_0\right)^{-\frac{n}{2}} \int_{\Gamma^s} e^{-\frac{|\xX-\xX_0|^2}{4t_0}} d\mathcal{H}^n.
\end{equation}

The goal of this appendix is to show that there is an $\epsilon_2\in (0,\epsilon_1)$ sufficiently small so that if $s\neq 0$ and $|s|\leq\epsilon_2$, then 
\begin{equation} \label{DecreaseEntEqn}
\lambda [\Gamma^s]\equiv\sup_{\xX_0,t_0} G(\xX_0,t_0,s)<G(\OO,1,0)=\lambda [\Sigma].
\end{equation}
In order to establish \eqref{DecreaseEntEqn}, we follow closely the proof of \cite[Theorem 0.15]{CM} and thus it suffices to show:

\begin{prop}
The following properties hold:
\begin{enumerate}
\item \label{DEPi}
$G$ has a strict local maximum at $(\OO,1,0)$;

\item \label{DEPii} 
The restriction of $G$ to $\Sigma$, i.e., $G(\xX_0,t_0,0)$, has a strict global maximum at $(\OO,1)$;

\item \label{DEPiii}
$|\partial_s G|$ is uniformly bounded on compact sets;

\item \label{DEPiv}
$G(\xX_0,t_0,s)$ is strictly less than $G(\OO,1,0)$ whenever $|\xX_0|$ is sufficiently large;

\item \label{DEPv}
$G(\xX_0,t_0,s)$ is strictly less than $G(\OO,1,0)$ whenever $|\log t_0|$ is sufficiently large.
\end{enumerate}
\end{prop}

\begin{proof}
Observe that, by \cite[Proposition 3.6]{CM}, the gradient of $G$ vanishes at $(\OO,1,0)$. Hence, Item \eqref{DEPi} will follow if we show that the Hessian of $G$ at $(\OO,1,0)$ is negative definite. That is, if, for any $a,b\in\mathbb{R}$ and $\yY\in\mathbb{R}^{n+1}$ so that $(a,b,\yY)\neq (0,0,\OO)$,
\begin{equation}
\partial_{ss}|_{s=0}G(s\yY,1+as,bs)<0.
\end{equation}
To see this, we first use \cite[Theorem 4.14]{CM} to conclude that, at $s=0$,
\begin{equation} \label{2ndVarEqn}
\begin{split}
\partial_{ss}G(s\yY,1+as,s)=(4\pi)^{-\frac{n}{2}} & \int_{\Sigma} \left(-f L_\Sigma f-a^2H^2_\Sigma-\frac{\left(\yY\cdot\nN_\Sigma\right)^2}{2}\right) e^{-\frac{|\xX|^2}{4}} d\mathcal{H}^n\\
& +(4\pi)^{-\frac{n}{2}} \int_{\Sigma} \left(2afH_\Sigma+f \yY\cdot\nN_\Sigma\right) e^{-\frac{|\xX|^2}{4}} d\mathcal{H}^n.
\end{split}
\end{equation}
Note that $f$, $H_\Sigma$, and $\yY\cdot\nN_\Sigma$ are eigenfunctions corresponding to different eigenvalues of $L_\Sigma$. As $\Sigma$ is asymptotically conical and $f$ satisfies \eqref{EigenBndEqn} and \eqref{EigenDerEqn}, it follows from \cite[Corollary 3.10]{CM} that $f$ is orthogonal to both $H_\Sigma$ and $\yY\cdot\nN_\Sigma$ in $L^2_w(\Sigma)$. This implies that the last integration in the right hand side of \eqref{2ndVarEqn} is zero and so, at $s=0$, 
\begin{equation}
\begin{split}
\partial_{ss} G(s\yY,1+as,s) & =(4\pi)^{-\frac{n}{2}} \int_{\Sigma} \left(-f L_\Sigma f-a^2H^2_\Sigma-\frac{(\yY\cdot\nN_\Sigma)^2}{2}\right) e^{-\frac{|\xX|^2}{4}} d\mathcal{H}^n\\
& \leq \mu (4\pi)^{-\frac{n}{2}} \int_{\Sigma} f^2 e^{-\frac{|\xX|^2}{4}} d\mathcal{H}^n<0.
\end{split}
\end{equation}
Hence, for any $b\in\mathbb{R}\setminus\{0\}$, the second derivative $\partial_{ss}|_{s=0}G(s\yY,1+sa,bs)<0$. Finally, appealing to \cite[Theorem 4.14]{CM} again,
\begin{equation}
\partial_{ss}\vert_{s=0} G(s\yY,1+sa,0)=(4\pi)^{-\frac{n}{2}} \int_{\Sigma} \left(-a^2H^2_\Sigma-\frac{(\yY\cdot\nN_\Sigma)^2}{2}\right) e^{-\frac{|\xX|^2}{4}} d\mathcal{H}^n<0,
\end{equation}
where the inequality follows from the fact that $\Sigma\in\mathcal{ACS}^*_n$ and so $\Sigma$ can neither split off a line isometrically (and so is non-flat) nor can $H_\Sigma$ cannot vanish identically. Hence, we conclude that the Hessian of $G$ at $(\OO,1,0)$ is negative definite as claimed. In particular, there is an $\epsilon_3\in (0,\epsilon_1)$ so that 
\begin{equation} \label{LocalMaxEqn}
G(\xX_0,t_0,s)<G(\OO,1,0)\quad\mbox{if $0<|\xX_0|^2+(\log t_0)^2+s^2<\epsilon_3^2$}.
\end{equation}

Next, \cite[Lemma 7.10]{CM} implies that $\lambda [\Sigma]=G(\OO,1,0)$, and there is an $\alpha_0>0$ so that
\begin{equation} \label{GlobalMaxEqn}
G(\xX_0,t_0,0)<G(\OO,1,0)-\alpha_0 \quad\mbox{if $|\xX_0|^2+(\log t_0)^2>\frac{\epsilon_3^2}{4}$}, 
\end{equation}
which gives Item \eqref{DEPii}.

As $\lambda [\Sigma]<\infty$, $\Sigma$ has bounded area ratios. Since $\Sigma$ has uniformly bounded second fundamental form, combining with \eqref{EigenBndEqn} and \eqref{EigenDerEqn}, it follows from \cite[(2.34)]{WaRigid} that
\begin{equation} \label{AreaEqn}
\sup_{|s|\leq\epsilon_1}\sup_{\xX\in\mathbb{R}^{n+1},r>1} r^{-n}\mathcal{H}^n(\Gamma^s\cap B_r(\xX))<C(\Sigma,C_0,C_1).
\end{equation}
Furthermore, as $\Sigma$ is asymptotically conical, by \eqref{EigenBndEqn} and \eqref{EigenDerEqn}, there is a $C_H$ depending only on $C_0,C_1,C_2$, and $\Sigma$ so that
\begin{equation} \label{MeanCurvBnd}
\sup_{|s|\leq\epsilon_1}\sup_{\Gamma^s} H_{\Gamma^s}^2 \leq C_H.
\end{equation}
Then, using the first variation formula \cite[Lemma 3.1]{CM}, we argue as in \cite[(7.5)]{CM} to get that
\begin{equation} 
\partial_{t_0} G(\xX_0,t_0,s)\geq -(4\pi t_0)^{-\frac{n}{2}}\int_{\Gamma^s}\frac{H^2_{\Gamma^s}}{4} e^{-\frac{|\xX-\xX_0|^2}{4t_0}} d\mathcal{H}^n\geq -\frac{C_H}{4} G(\xX_0,t_0,s).
\end{equation}
Thus, for $0<t_0\leq\bar{t}$,
\begin{equation} \label{1stVartEqn}
G(\xX_0,t_0,s)\leq e^{\frac{C_H\bar{t}}{4}} G(\xX_0,\bar{t},s).
\end{equation} 
In particular, choosing $\bar{t}=1$, this together with \eqref{AreaEqn} implies that 
\begin{equation} \label{AreaRatioEqn}
\sup_{|s|\leq\epsilon_1}\sup_{\xX\in\mathbb{R}^{n+1},r>0} r^{-n}\mathcal{H}^n(\Gamma^s\cap B_r(\xX))<C_M,
\end{equation}
where $C_M$ depends only on $n$, $C$, and $C_H$. 

An easy consequence of \eqref{AreaRatioEqn} is that given $\alpha_1>0$ small, there is an $R>0$ so that 
\begin{equation} \label{UniformSmallEqn}
\sup_{|s|\leq\epsilon_1} \sup_{\xX_0,t_0} (4\pi t_0)^{-\frac{n}{2}} \int_{\Gamma^s\setminus B_{R\sqrt{t_0}}(\xX_0)} \frac{|\xX-\xX_0|}{\sqrt{t_0}} e^{-\frac{|\xX-\xX_0|^2}{4t_0}} d\mathcal{H}^n<\frac{\alpha_1}{4}.
\end{equation}
Invoking the first variation formula again, we get that
\begin{equation} \label{1stVarsEqn}
\partial_s G(\xX_0,t_0,s)=(4\pi t_0)^{-\frac{n}{2}} \int_{\Gamma^s} f_{\Gamma^s} \left(H_{\Gamma^s}-\frac{(\xX-\xX_0)\cdot\nN_{\Gamma^s}}{2t_0}\right) e^{-\frac{|\xX-\xX_0|^2}{4t_0}} d\mathcal{H}^n.
\end{equation}
Let $\mathcal{K}$ be any compact set in $\mathbb{R}^{n+1}\times\mathbb{R}^+\times [-\epsilon_1,\epsilon_1]$. There is an $\bar{R}$ so that $B_{\bar{R}}$ contains $B_{R\sqrt{t_0}}(\xX_0)$ for every $(\xX_0,t_0,s)\in\mathcal{K}$. Thus, with \eqref{EigenBndEqn} and \eqref{MeanCurvBnd}, it follows from \eqref{UniformSmallEqn} that the integration in \eqref{1stVarsEqn} restricted on $\mathbb{R}^{n+1} \setminus B_{\bar{R}}$ will be uniformly bounded on $\mathcal{K}$. Clearly, the integration in \eqref{1stVarsEqn} restricted on $B_{2\bar{R}}$ is continuous in all three variables $\xX_0$, $t_0$, and $s$. Hence we conclude that $\partial_s G$ is uniformly bounded on $\mathcal{K}$, proving Item \eqref{DEPiii}.

To see Item \eqref{DEPiv}, observe that, by Item \eqref{PISPii} of Proposition \ref{PerturbedInitialSurfaceProp}, for any given $\alpha_1>0$ small, there is an $R^\prime>1$ so that if $|\xX_0|>R^\prime$, $0<t_0 \leq 1$, and $|s|\leq\epsilon_1$, then either $(\Gamma^s-\xX_0)/\sqrt{t_0}$ is contained outside $B_R$, or $(\Gamma^s-\xX_0)/\sqrt{t_0}$ in $B_{2R}$ is a connected small $C^1$ graph over some hyperplane. This, together with \eqref{UniformSmallEqn}, implies that $G(\xX_0,t_0,s)<1+\alpha_1/2$. 

Next we study the region where $t_0>1$ and $|\xX_0|$ sufficiently large. Set $\rho=1/\sqrt{t_0}<1$ and $\yY_0=\xX_0/\sqrt{t_0}$. By \eqref{AreaRatioEqn}, given $\alpha_1>0$ small, there is an $r_1>0$ small so that for all $\yY_0\in\mathbb{R}^{n+1}$, $\rho<1$, and $|s|\leq\epsilon_1$,
\begin{equation} \label{UniformSmall1Eqn}
(4\pi)^{-\frac{n}{2}} \int_{\rho\Gamma^s\cap B_{r_1}} e^{-\frac{|\xX-\yY_0|^2}{4}} d\mathcal{H}^n<\frac{\alpha_1}{4}.
\end{equation}
Observe that $\rho\Gamma^s$ is the normal graph of the function $s\rho f(\rho^{-1}\cdot)$ on $\rho\Sigma$. As $\rho\Sigma\to\mathcal{C}(\Sigma)$ in $C^\infty_{loc}(\mathbb{R}^{n+1}\setminus\set{\OO})$ as $\rho\to 0$ and the link $\mathcal{L}(\Sigma)$ is smooth embedded, it follows from \eqref{EigenBndEqn} and \eqref{EigenDerEqn} that there is a $C_A(r_1)$ so that
\begin{equation}
\sup_{|s|\leq\epsilon_1, \rho<1}\sup_{\rho\Gamma^s\setminus B_{r_1}} \abs{A_{\rho\Gamma^s}}<C_A(r_1).
\end{equation}
Furthermore, together with \eqref{AreaRatioEqn}, \eqref{UniformSmallEqn} and \eqref{UniformSmall1Eqn}, we get that, if $\yY_0\in\mathbb{R}^{n+1}$, $\rho<1$, and $|s|\leq\epsilon_1$, then
\begin{equation} \label{TimeNotSmallEqn}
\begin{split}
F[\rho\Gamma^s-\yY_0] & \leq F[\rho\Sigma-\yY_0]+\frac{\alpha_1}{2}+(4\pi)^{-\frac{n}{2}}\int_{\rho\Gamma^s\cap\left(B_R(\yY_0)\setminus B_{r_1}\right)} e^{-\frac{|\xX-\yY_0|^2}{4}} d\mathcal{H}^n \\
& -(4\pi)^{-\frac{n}{2}}\int_{\rho\Sigma\cap\left(B_R(\yY_0)\setminus B_{r_1}\right)} e^{-\frac{|\xX-\yY_0|^2}{4}} d\mathcal{H}^n \\
& \leq F[\rho\Sigma-\yY_0]+\frac{\alpha_1}{2}+C_E,
\end{split}
\end{equation}
where $C_E$ depends on $C_0, C_1, \epsilon_1, r_1,C_M$, and $C_A$ with $C_E\to 0$ as $\epsilon_1\to 0$. 

Now let us choose $\alpha_1=\min\set{\alpha_0/4,2(\lambda [\Sigma]-1)/5}$ and possibly shrink $\epsilon_1$ so that $C_E<\alpha_0/4$. Hence, by \eqref{GlobalMaxEqn}, we conclude that for all $|\xX_0|>\bar{R}^{\prime}=\max\set{R^\prime,\frac{\epsilon_3}{2}}$
\begin{equation} \label{SpaceLargeEqn}
G(\xX_0,t_0,s)<\lambda [\Sigma]-\frac{\alpha_0}{2}.
\end{equation}
  
Finally, in view of \eqref{TimeNotSmallEqn}, we only need to deal with the case that $t_0$ approaches to $0$ and $|\xX_0|\leq\bar{R}^\prime$ so to complete the proof of Item \eqref{DEPv}. We choose $\bar{t}$ sufficiently small so that
\begin{equation}
e^{\frac{C_H\bar{t}}{4}} \left(\lambda[\Sigma]-\alpha_0\right)<\lambda [\Sigma]-\frac{\alpha_0}{2}.
\end{equation}
Thus, further shrinking $\epsilon_1$ if necessary, it follows from \eqref{1stVartEqn}, \eqref{GlobalMaxEqn} and Item \eqref{DEPiii} that for all $0<t_0<\bar{t}$, $|\xX_0|\leq\bar{R}^\prime$, and $|s|\leq\epsilon_1$, we have that $G(\xX_0,t_0,s)<\lambda [\Sigma]-\alpha_0/4$.
\end{proof}

\begin{acknowledgement*}
We would like to thank Detang Zhou for many constructive comments, in particular for explaining to us an alternative proof of the half-space theorem for self-shrinkers and the Rellich-Kondrachov type theorem of weighted Sobolev spaces. 
\end{acknowledgement*}

\end{document}